\documentclass[english,final]{amsart}
\usepackage[english]{babel}

\usepackage[style=numeric,backend=biber]{biblatex}
\usepackage[colorlinks=true,unicode]{hyperref}
\usepackage{amsthm,amsmath,amssymb,amscd,amsxtra,amsgen,enumitem,verbatim,ulem}
\usepackage[all]{xy} \CompileMatrices
\setenumerate[1]{label=(\theenumi)}
\newtheorem*{thma}{ \bf Theorem \ref{purity-br-loc}}
\begin{document}

\newcommand{\qr}[1]{\eqref{#1}} 
\newcommand{\var}[2]{\operatorname{var}_{#1}^{#2}}
\newcommand{\secref}[1]{\ref{#1}}
\newcommand{\varn}[1]{\var{#1}{}}

\newcommand{\ie}{{\it i.e.}\xspace} 
\newcommand{\eg}{{\it e.g.}\xspace}
\newcommand{\nc}{\newcommand}
\renewcommand{\frak}{\mathfrak}
\providecommand{\cal}{\mathcal}

\renewcommand{\bold}{\mathbf}
%
%
%
%
\newcommand{\gparens}[3]{{\left#1 #2 \right#3}}
\newcommand{\parens}[1]{\gparens({#1})} 
\newcommand{\brackets}[1]{\gparens\{{#1}\}} 
\newcommand{\hakparens}[1]{\gparens[{#1}]} 
\newcommand{\angleparens}[1]{\gparens\langle{#1}\rangle} 
\newcommand{\floor}[1]{\gparens\lfloor{#1}\rfloor} 
\newcommand{\ceil}[1]{\gparens\lceil{#1}\rceil}
\newcommand*{\Setof}[1]{\,\brackets{#1}\,} 
\newcommand*{\norm}[2][]{\gparens\|{#2}\|#1}
\newcommand*{\normsq}[2][]{\gparens\|{#2}\|^2#1}
%

\numberwithin{equation}{section}

\newcommand{\theoremname}{Theorem.}
\newcommand{\corollaryname}{Corollary.}
\newcommand{\lemmaname}{Lemma.}
\newcommand{\propositionname}{Proposition.}
\newcommand{\conjecturename}{Conjecture.}
\newcommand{\definitionname}{Definition.}
\newcommand{\examplename}{Example.}
\newcommand{\remarkname}{Remark.}
\newcommand{\pfname}{Proof.}

\newenvironment{pf}{\vskip-\lastskip\vskip\medskipamount{\it\pfname}}%
                      {$\square$\vskip\medskipamount\par}

\newenvironment{pfof}[1]{\vskip-\lastskip\vskip\medskipamount{\it
    Proof of #1.}}%
                      {$\square$\vskip\medskipamount\par}

\newtheorem{thm}{Theorem}[section]
\newtheorem{theorem}[thm]{Theorem}

\newtheorem{thmnn}{Theorem}
\renewcommand{\thmnn}{}

\newtheorem{cor}[thm]{Corollary}
\newtheorem{corollary}[thm]{Corollary}
\newtheorem{cornn}{Corollary}
 \renewcommand{\thecornn}{}

\newtheorem{prop}[thm]{Proposition}
\newtheorem{proposition}[thm]{Proposition}

\newtheorem{propnn}{Proposition}
\renewcommand{\thepropnn}{}

\newtheorem{lemma}[thm]{Lemma} 
\newtheorem{lemmat}[thm]{Lemma}
\newtheorem{lemmann}{Lemma}
\renewcommand{\thelemmann}{}

\theoremstyle{definition}
\newtheorem{defn}[thm]{Definition}
\newtheorem{definition}[thm]{Definition}
\newtheorem{defnn}{Definition}
\renewcommand{\thedefnn}{}
\newtheorem{conj}[thm]{Conjecture}
\newtheorem{axiom}{Axiom}

\theoremstyle{definition}
\newtheorem{exercise}{Exercise}[subsection]
\newtheorem{exercisenn}{Exercise}
\renewcommand{\theexercisenn}{}
\newtheorem{remark}[thm]{Remark}
\newtheorem{remarks}[thm]{Remarks}
\newtheorem{remarknn}{Remark}
\renewcommand{\theremarknn}{}
\newtheorem{example}[thm]{Example} 
\newtheorem{examples}[thm]{Examples} 
\newtheorem{examplenn}{Example}[subsection]
\newtheorem{blanktheorem}{}[subsection]
\newtheorem{question}[thm]{Question}
\renewcommand{\theexamplenn}{}


\nc{\Theorem}[1]{Theorem~{#1}}
\nc{\Th}[1]{({\sl Th.}~#1)}
\nc{\Thd}[2]{({\sl Th.}~{#1} {#2})}
\nc{\Theorems}[2]{Theorems~{#1} and ~{#2}}
\nc{\Thms}[2]{({\it Thms. ~{#1} and ~{#2}})}
\nc{\Lemmas}[2]{Lemma~{#1} and ~{#2}}

\nc{\manga}[6]{({\it Thms. ~ #1, ~ #2, ~ #3,\\ ~ #4, ~ #5, ~ #6})}

\nc{\Prop}[1]{({\sl Prop.}~{#1})}
\nc{\Proposition}[1]{Proposition~{#1}}
\nc{\Propositions}[2]{Propositions~{#1} and ~{#2}}
\nc{\Props}[2]{({\sl Props.}~{#1} and ~{#1})}
\nc{\Cor}[1]{({\sl Cor.}~{#1})}
\nc{\Corollary}[1]{Corollary~{#1}}
\nc{\Corollaries}[2]{Corollaries~{#1} and ~{#2}}
\nc{\Definition}[1]{Definition~{#1}}
\nc{\Defn}[1]{({\sl Def.}~{#1})}
\nc{\Lemma}[1]{Lemma~{#1}} 
\nc{\Lem}[1]{({\sl Lem.} ~{#1})} 
\nc{\Eq}[1]{equation~({#1})}
\nc{\Equation}[1]{Equation~({#1})}
\nc{\Section}[1]{Section~{#1}}
\nc{\Sections}[1]{Sections~{#1}}
\nc{\Sec}[1]{({\sl Sec.} ~{#1})} 
\nc{\Chapter}[1]{Chapter~{#1}}
\nc{\Chapt}[1]{({\sl Ch.}~{#1})}

\nc{\Ex}[1]{{\sl Ex.}~{#1}}
\nc{\Exa}[1]{{\sl Example}~{#1}}
\nc{\Example}[1]{{\sl Example}~{#1}}
\nc{\Examples}[1]{{\sl Examples}~{#1}}
\nc{\Exercise}[1]{{\sl Exercise}~{#1}}

\nc{\Rem}[1]{({\sl Rem.}~{#1})}
\nc{\Remark}[1]{{\sl Remark}~{#1}}
\nc{\Remarks}[1]{{\sl Remarks}~{#1}}
\nc{\Note}[1]{{\sl Note}~{#1}}

\nc{\Conjecture}[1]{Conjecture~{#1}}
\nc{\Claim}[1]{Claim~{#1}}

\nc \Proof{{  \it Proof. }}

\nc{\xmu}{\mu}
\nc{\w}{\omega}
\nc{\xv}{\mbox{\boldmath$x$}}
\nc{\uv}{\mbox{\boldmath$u$}}
\nc{\xiv}{\mbox{\boldmath$\xi$}}
\nc{\bbeta}{\mbox{\boldmath$\beta$}}
\nc{\balpha}{\mbox{\boldmath$\alpha$}}
\nc{\bgamma}{\mbox{\boldmath$\gamma$}}
\nc{\bdelta}{\mbox{\boldmath$\delta$}}
\nc{\bepsilon}{\mbox{\boldmath$\epsilon$}}

\newcommand{\ZZ}{{\mathbb Z}}
\newcommand{\RR}{{\mathbb R}} 

\nc \Ab{{\ensuremath{\bold A}}}
\nc \ab{{\ensuremath{\bold a}}}
\nc \bb{{\ensuremath{\bold b}}}
\nc \cb{{\ensuremath{\bold c}}}
\nc \Bb{{\ensuremath{\bold B}}}
\nc \Gb{{\ensuremath{\bold G}}}
\nc \Qb{{\ensuremath{\bold Q}}}
\nc \Rb{{\ensuremath{\bold R}}} \nc \Cb{{\ensuremath{\bold C}}} 
\nc \Eb{{\ensuremath{\bold E}}}
\nc \eb{{\ensuremath{\bold e}}}
\nc \Db{{\ensuremath{\bold D}}}
\nc \Fb{{\ensuremath{\bold F}}}
\nc \ib{{\ensuremath{\bold i}}}
\nc \jb{{\ensuremath{\bold j}}}
\nc \kb{{\ensuremath{\bold k}}}
\nc \nb{{\ensuremath{\bold n}}}
\nc \rb{{\ensuremath{\bold r}}}
\nc \Pb{{\ensuremath{\bold P}}}
\nc \pb{{\ensuremath{\bold p}}}
\nc \SPb{{\ensuremath{\bold {SP}}}}
\nc \Zb{{\ensuremath{\bold Z}}} 
\nc \zb{{\ensuremath{\bold z}}} 
\nc \gb{{\ensuremath{\bold g}}} 
\nc \fb{{\ensuremath{\bold f}}} 
\nc \ub{{\ensuremath{\bold u}}} 
\nc \vb{{\ensuremath{\bold v}}} 
\nc \yb{{\ensuremath{\bold y}}} 
\nc \xb{{\ensuremath{\bold x}}} 
\nc \xib{{\ensuremath{\bold \xi}}} 
\nc \Nb{{\ensuremath{\bold N}}} 
\nc \Hb{{\ensuremath{\bold H}}} 
\nc \wb{{\ensuremath{\bold w}}} 
\nc \Wb{{\ensuremath{\bold W}}} 
\nc \syz{{\mathbf {syz}}}
\nc \bnoll{{\ensuremath{\bold 0}}} 

\nc \mf{\frak m} \nc \mh{\hat{\m}} 
\nc \nf{\frak n}
\nc \Of{\frak O}
\nc \rf{\frak r}
\nc \mufr{{\mathbf \mu}}
\nc \hf{\frak h} 
\nc \qf{\frak q} 
\nc \bfr{\frak b} 
\nc \kfr{\frak k} 
\nc \pfr{\frak p} 
\nc \af{\frak a }
\nc \cf{\frak c }
\nc \sfr{\frak s} 
\nc \ufr{\frak u} 
\nc \g{\frak g} 
\nc \gA{\g_{\Ao}} 
\nc \lfr{\frak l}
\nc \afr{\frak a}
\nc \gfh{\hat {\frak g}}
\nc \gl{\frak { gl }}
\nc \Sl{\frak {sl}}
\nc \SU{\frak {SU}}
\nc{\Homf}{\frak{Hom}}

\newcommand{\on}{\operatorname}
\nc\hankel{\on {Hankel}}
\nc\row{\on {row\ }}
\nc\nullity{\on {nullity }}
\nc\col{\on {col\ }}
\nc\rowm{\on {Row \ }}
\nc\loc{\on {lc \ }}
\nc\nullo{\on {null\ }}
\nc\Nul{\on {Nul\ }}
\nc \Ann {\on {Ann }}
\nc \Ass {\on {Ass \ }}
\nc \Coker {\on {Coker}}
\nc \Co{\on C}
\nc \Homo{\on {Hom}}
\nc \Ker {\on {Ker}}
\nc \omod{\on {mod}}
\nc \No {\on N}
\nc \NN {\on {NN}}
\nc \NGo {\on {NG}}
\nc \Oo {\on O}
\nc \ch {\on {ch}}
\nc \rko {\on {rk}}
\nc \Sing {\on {Sing\ }}
\nc \Reg {\on {Reg}}
\nc \CoI {\on {CI}}
\nc \CoM {\on {CM}}
\nc \Gor {\on {Gor}}
\nc \Type {\on {Type}}
\nc \can {\on {can}}
\nc \Top {\on {T}}
\nc \Tr {\on {Tr}}
\nc \rel {\on {rel}}
\nc \tr {\on {tr}}
\nc \sgn {\on {sgn }}
\nc \trdeg {\on {tr.deg}}
\nc \codim {\on {codim }}
\nc \coht {\on {coht}}
\nc \divo {\on {div \ }}
\nc \coh {\on {coh}}
\nc \Clo {\on {Cl}}
\nc \embdim{\on {embdim}}
\nc \ed{\on {ed}}
\nc \embcodim{\on {embcodim  }}
\nc \qcoh {\on {qcoh}}
\nc \grad {\on {grad}\ }
\nc \grade {\on {grade}}
\nc \hto {\on {ht}}
\nc \depth {\on {depth}}
\nc \prof {\on {prof}}
\nc \reso{\on {res}}
\nc \ind{\on {ind}}
\nc \prodo{\on {prod}}
\nc \coind{\on {coind}}
\nc \Con{\on {Con}}
\nc \Crit{\on {Crit}}
\nc \Der{\on {Der}}
\nc \Char{\on {Char}}
\nc \Ch{\on {Ch}}

\nc \Ext{\on {Ext}}
\nc \Eo{\on {E}}
\nc \End{\on {End}}
\nc \ad{\on {ad}}
\nc \Ad{\on {Ad}}
\nc \gr{\on {gr}}
\nc \Fo{\on {F}}
\nc \Gr{\on {Gr}}
\nc \Go{\on {G}}
\nc \GFo{\on {GF}}
\nc \Glo{\on {Gl}}
\nc \PGlo{\on {PGl}}
\nc \Ho{\on {H}}
\nc \CMo{\on {\CM}}
\nc \SCM{\on {SCM}}
\nc \hol{\on {hol}}
\nc{\sgd}{\on{sgd}}
\nc \supp{\on {supp}}
\nc \ssupp{\on {s-supp}}
\nc \singsupp{\on {singsupp}}
\nc \msupp{\on {msupp}}
\nc \spec{\on {spec}}
\nc \spano{\on {span }}
\nc \Span{\on {Span }}
\nc \Max{\on {Max}}
\nc \Mat{\on {Mat}}
\nc \Min{\on {Min}}
\nc \nil{\on {nil}}
\nc \Mod{\on {Mod}}
\nc \Rad {\on {Rad}}
\nc \rad {\on {rad}}
\nc \rank {\on {rank}}
\nc \range {\on {range}}
\nc \Slo{\on {SL}}
\nc \soc {\on {soc}}
\nc \Irr {\on {Irr}}
\nc \Reo {\on {Re}}
\nc \Imo {\on {Im}}
\nc \SSo{\on {SS}}
\nc \lub{\on {lub}}
\nc \gldim{\on {gl.d.}}
\nc \length{\on {length}}
\nc \pdo{\on {p.d.}} 
\nc \fdo{\on {f.d.}} 
\nc \ido{\on {i.d.}} 
\nc \dSSo{\dot {\SSo}}
\nc \So{\on S}
\nc \Io{\on I}
\nc \Jo{\on J}
\nc \jo{\on j}
\nc \Ko{\on K}
\nc \PBW{\Ac_{PBW}}
\nc \Ro{\on R}
\nc \To{\on T}
\nc \Ao{\on A}

\nc \Do{{\on D}}
\nc \Bo{\on B}
\nc \Po{\on P}
\nc \Qo{\on Q}
\nc \Zo{\on Z}
\nc \U{\on U}
\nc \wt{\on {wt}}
\nc \Uh{\hat {\U}}
\nc \T{\on T}
\nc \Lo{\on L}
\nc{\dop}{\on d}
\nc{\eo}{\on e}
\nc{\ado}{\on{ad}}
\nc{\Tot}{\on{Tot}}
\nc{\Aut}{\on{Aut}}
\nc{\sinc}{\on {sinc}}

%
%
\nc{\overrightleftarrows}[2]{\overset{#1}{\underset{#2}{\rightleftarrows}}}

\nc{\CCF}{\cal{CF}}
\nc{\CDF}{\cal{DF}}
\nc{\CHC}{\check{\cal C}}

\nc{\Cone}{\on{Cone}}
\nc{\dec}{\on{dec}}
\nc{\Diff}{\on{Diff}}
\nc{\dirlim}{\underset{\to}{\on{lim}}}
\nc{\dpar}{\partial}
\nc{\GL}{\on{GL}}
\nc{\CGr}{\cal{G}r}
\nc{\pr}{\on{pr}}
\nc{\semid}{|\!\!\!\times}
\nc{\Hom}{\on{Hom}}
\nc \RHom{\on {RHom}}

\nc \Proj{\mathrm {Proj\ }}
\nc \proj{\mathrm {proj}}
\nc{\Id}{\on{Id}}
\nc{\id}{\on{id}}
\nc{\Ima}{\on{Im}}
\nc{\invtimes}{\underset{\gets}{\otimes}}
\nc{\invlim}{\underset{\gets}{\on{lim}}}
\nc{\Lie}{\on{Lie}}
\nc{\re}{\on{Re }}
\nc{\Pic}{\on{Pic }}
\nc{\LPic}{\on{LPic }}
\nc{\Sch}{\on{Sch}}
\nc{\Sh}{\on{Sh}}
\nc{\Set}{\on{Set}}
\nc{\spo}{\on{sp\  }}
\nc{\Spec}{\on{Spec}}
\nc{\mSpec}{\on{mSpec}}
\nc{\Specb}{\bold {Spec}\ }
\nc{\Projb}{\bold {Proj}}
\nc{\Specan}{\on{Specan}}
\nc{\Spo}{\on{Sp}}
\nc{\Spf}{\on{Spf}}
\nc{\sym}{\on{sym}}
\nc{\symm}{\on{symm}}
\nc{\rop}{\on{r}}
\nc{\Td}{\on{Td}}
\nc{\Tor}{\on{Tor}}


\nc{\Artin}{\cal{A}rtin}
\nc{\Dgcoalg}{\cal{D}gcoalg}
\nc{\Dglie}{\cal{D}glie}
\nc{\Ens}{\cal{E}ns}
\nc{\Fsch}{\cal{F}sch}
\nc{\Groupoids}{\cal{G}roupoids}
\nc{\Holie}{\cal{H}olie}
\nc{\Mor}{\cal{M}or}

\nc{\CF}{\ensuremath{\cal{F}}}
\nc \Kc{{\ensuremath{\cal K}}}
\nc \Lc{{\ensuremath{\cal L}}}
\nc \lcc{{\mathcal l}} 
\nc \CC{{\ensuremath{\cal C}}} 
\nc \Cc{{\ensuremath {\cal C}}}
\nc \Pc{{\ensuremath{\cal P}}}
\nc \Dc{\ensuremath{\mathcal D}}
\nc \Ac{{\ensuremath{\cal A}}} 
\nc \Bc{{\ensuremath{\cal B}}}
\nc \Ec{{\ensuremath{\cal E}}}
\nc \Fc{{\ensuremath{\cal F}}}
\nc \Mcc{{\ensuremath{\cal M}}} 
\nc \hM{\hat{\Mcc}} 
\nc \bM{\bar {\Mcc}} 
\nc\hbM{\hat{\bar \Mcc}}  
\nc \Nc{{\ensuremath{\cal N}}}
\nc \Hc{{\ensuremath{\cal H}}} 
\nc \Ic{{\ensuremath{\cal I}}} 
\nc \Oc{\ensuremath{{\cal O}}}
\nc \qc{\ensuremath{{\Cal q}}}
\nc \Och{\hat{\cal O}} 
\nc \Sc{{\ensuremath{{\cal S}}}}
\nc \Tc{\ensuremath{{\cal T}}} 
\nc \Vc{{\ensuremath{{\cal V}}}} 
\nc{\CA}{{\ensuremath{{\cal A}}}}
\nc{\CB}{{\ensuremath{{\cal B}}}}
\nc{\C}{{\ensuremath{{\cal F}}}}
\nc{\Gc}{{\ensuremath{{\cal G}}}}
\nc{\CH}{\ensuremath{\mathcal H}}
\nc{\CI}{{\ensuremath{{\cal I}}}}
\nc{\CM}{{\ensuremath{{\cal M}}}}
\nc{\CN}{{\ensuremath{{\cal N}}}}
\nc{\CO}{{\ensuremath{{\cal O}}}}
\nc{\Rc}{{\ensuremath{{\cal R}}}}
\nc{\CT}{{\ensuremath{\mathcal T}}}
\nc{\CU}{\ensuremath{{\cal U}}}
\nc{\CV}{\ensuremath{{\cal V}}}
\nc{\CZ}{\ensuremath{{\cal Z}}}
\nc{\Homc}{\ensuremath{{\cal {Hom}}}}


\nc{\fa}{\frak{a}}
\nc{\fA}{\frak{A}}
\nc{\fg}{\frak{g}}
\nc{\fh}{\frak{h}}
\nc{\fI}{\frak{I}}
\nc{\fK}{\frak{K}}
\nc{\fm}{\frak{m}}
\nc{\fP}{\frak{P}}
\nc{\fS}{\frak{S}}
\nc{\ft}{\frak{t}}
\nc{\fX}{\frak{X}}
\nc{\fY}{\frak{Y}}

\nc{\bF}{\bar{F}}
\nc{\bCP}{\bar{\cal{P}}}
\nc{\bm}{\mbox{\bf{m}}}
\nc{\bT}{\mbox{\bf{T}}}
\nc{\hB}{\hat{B}}
\nc{\hC}{\hat{C}}
\nc{\hP}{\hat{P}}
\nc{\htest}{\hat P}


\nc{\nen}{\newenvironment}
\nc{\ol}{\overline}
\nc{\ul}{\underline}
\nc{\ra}{\to}
\nc{\lla}{\longleftarrow}
\nc{\lra}{\longrightarrow}
\nc{\Lra}{\Longrightarrow}
\nc{\Lla}{\Longleftarrow}
\nc{\Llra}{\Longleftrightarrow}
\nc{\hra}{\hookrightarrow}
\nc{\iso}{\overset{\sim}{\lra}}

\nc{\dsize}{\displaystyle}
\nc{\sst}{\scriptstyle}
\nc{\tsize}{\textstyle}
\nen{exa}[1]{\label{#1}{\bf Example.\ } }{}


\nen{rem}[1]{\label{#1}{\em Remark.\ } }{}
\nen{note}[1]{\label{#1}{\em Note.\ } }{}

\title{Purity of branch and critical locus} \author{Rolf
  K\"allstr\"om} \address{ Department of Mathematics,University of
  G\"avle, 100 78 G\"avle, Sweden} \date{\today} \email{rkm@hig.se}

\maketitle
\begin{abstract} To a
  dominant morphism $X/S \to Y/S$ of N\oe therian integral $S$-schemes
  one has the inclusion $C_{X/Y}\subset B_{X/Y}$ of the critical locus
  in the branch locus of $X/Y$.  Starting from the notion of locally
  complete intersection morphisms, we give conditions on the modules
  of relative differentials $\Omega_{X/Y}$, $\Omega_{X/S}$, and
  $\Omega_{Y/S}$ that imply bounds on the codimensions of $ C_{X/Y}$
  and $ B_{X/Y}$.  These bounds generalise to a wider class of
  morphisms the classical purity results for finite morphisms by
  Zariski-Nagata-Auslander, and Faltings and Grothendieck, and van der
  Waerden's purity for birational morphisms.
\end{abstract}

\section{Introduction} In this paper\footnote{This is a small revision
  of the original publication, see \Remark{\ref{revised}} }
$\pi:X/S \to Y/S$ denotes a dominant morphism of N\oe therian integral
$S$-schemes, which is locally of finite type and of relative dimension
$d_{X/Y}$. Let $\Omega_{X/Y}$ be the sheaf of relative differentials, i.e.
\begin{displaymath}
   \Omega_{X/Y}=   \Coker (\pi^*(\Omega_{Y/S}) \to \Omega_{X/S}),
 \end{displaymath}
 and dually let $\Cc_{X/Y} =$ $ \Coker(d\pi: T_{X/S}\to T_{X/S\to
   Y/S})$ be the critical module of $\pi$, where $d\pi$ is the tangent
 morphism of $\pi$.  The {\it critical locus} \/of $\pi$ is the
 support $C_\pi$ of $\Cc_{X/Y}$, and the {\it branch locus} $B_\pi$ is
 the set of points $x$ where the stalk $\Omega_{X/Y,x}$ is not free;
 we abuse the terminology since $B_\pi$ is the set of ramification
 points as defined in \cite{EGA4:4} only when $\Omega_{X/Y}$ is
 torsion and hence $B_\pi = \supp \Omega_{X/Y}$.  These two subsets of
 $X$, satisfying $C_\pi\subset B_\pi$, exert much control on the
 morphism $\pi$.  If $B_\pi = \emptyset$, then $\pi$ is smooth (as
 defined in \cite[Def.  17.3.1]{EGA4:4}) if it is flat and generically
 smooth, and if moreover $\pi$ is finite and $Y$ is normal, then $\pi$
 is \'etale \cite[Sec.  4]{auslander-buchsbaum}; see
 \cite{kallstrom:zl} for a discussion of the relation between the
 branch locus and the non-smoothness locus.
 If $C_\pi = \emptyset$, and $\pi$ is either flat or $Y/S$ is smooth,
 then tangent vector fields on $Y$ lift (locally) to tangent vector
 fields on $X$, so according to Zariski's lemma ($\Char X=0$) the
 morphism $\pi$ is locally analytically trivial.  It is therefore a
 natural problem to find upper bounds on the codimensions of $B_\pi$
 and $C_\pi$, so that $B_\pi = \emptyset$ or $C_\pi = \emptyset$ can
 be controlled in low codimensions.  The best situation is when
 $\codim_X C_\pi \leq 1$ and $\codim_X B_\pi \leq 1$ (when nonempty),
 and we say that $C_\pi$ and $B_\pi$ are {\it pure} \/(of codimension
 $1$), respectively.  Let $F_i(M)$ denote the ith Fitting ideal of a
 module and the relative dimension $d_{X/Y}$ be defined as the Krull
 dimension of the generic fibre of $\pi$.  If $X/S$ and $Y/S$ are
 smooth there is a duality relation
 \begin{equation}\tag{$*$}\label{fitting-dual}
   F_i(\Cc_{X/Y})= F_{d_{X/Y}+i}(\Omega_{X/Y})
 \end{equation} \Prop{\ref{dual-crit}}, so in particular  $C_\pi=B_\pi$ in
 this case.   A simple notable fact is that $\codim_X C_\pi\leq 1$ if the image of
 the tangent morphism $ \Imo (d\pi)$ satisfies Serre's condition
 $(S_2)$, and that this holds in particular when $X$ satisfies $(S_2)$
 and $\pi$ is generically separably algebraic; hence  by
 \thetag{$*$}  $\codim_X B_\pi \leq 1$ when $X/S$ and $Y/S$ are smooth.
 In general we shall see that  
 $C_\pi$ is pure ``more often'' than $B_\pi$  \Th{\ref{purity of critical}}.     

 Our method of establishing purity results for $C_\pi$ and $B_\pi$ is
 by assuming a good behaviour of the modules $\Omega_{X/S}$ and
 $\Omega_{Y/S}$.  Say that $\pi$ is a {\it differentially complete
   intersection morphisms} (d.c.i.) at a point $x$ if the projective
 dimension $\pdo \Omega_{X/Y,x}\leq 1$.  This notion is inspired by a
 result due to Ferrand and Vasconcelos
 \cite{ferrand:suite,vasconcelos:normality} that in the case of
 generically smooth domains over a field, which, when extended to a
 relative situation, states that locally complete intersection
 morphisms $X/Y$ that are smooth at all the associated points in $X$
 are $d.c.i.$, and that the converse holds if $\pdo_{\Oc_{Y,\pi(x)}}
 \Oc_{X,x}<\infty $ for each point $x\in X$; we include a complete
 proof \Th{\ref{ferrand-vasconcelos}}.
 In \Section{\ref{diffcompletesection}} we work out some basic results
 for d.c.i. morphisms.  Not only is the notion of d.c.i. morphisms
 more general than that of locally complete intersection morphisms,
 but it is also in some respects easier to work with, since many
 proofs of basic properties are rather straightforward.  For example,
 we easily get a base-change theorem for locally complete intersection
 morphisms without flatness assumptions \Th{\ref{base-change:dci}},
 and also composition and decomposition properties for
 d.c.i. morphisms \Th{\ref{comp-dci}}. This can be compared to results
 by Avramov \cite[5.6, 5.7 ,5.11]{avramov:lci}, where composition and
 decomposition properties for locally complete intersection morphisms
 are established using deep characterisations of such morphisms in
 terms of the cotangent complex, and to get the corresponding base
 change theorem a type of flatness is required.

 Since the higher branch and critical loci $ B_\pi^{(i)}$ and $
 C^{(i)}_\pi$ ($B^{(0)}_\pi=B_\pi$ and $C^{(0)}_\pi= C_\pi$) we are
 interested in are defined by certain Fitting ideals, in order to
 achieve bounds on codimensions we relate heights of Fitting ideals to
 Euler characteristics and Betti numbers.  This is a general
 discussion that can be of some independent interest, based on bounds
 by Eagon and Northcott \cite{eagon-northcott:height}, and Eisenbud,
 Ulrich and Huneke \cite{eisenbud-huneke-ulrich:heights}.

 Let $\delta_{X/Y}$ be the relative smoothness defect, $\chi_2(M)$ a
 partial Euler characteristic, and $\beta_1(M)$ the first Betti
 number of a module (see \Section{\ref{sec-pur-general}} for precise
 definitions).  The following is an example of the type of bounds that
 we attain:
  \begin{thma}\label{intro-thm} Let $\pi : X/S\to Y/S$ be a
    generically smooth morphism of N\oe therian integral $S$-schemes
    such that $\Omega_{X/S}$ and $\Omega_{Y/S}$ are coherent. 
    \begin{enumerate}    \item 
  \begin{displaymath}  
    \codim_X^+B_\pi^{(i)}\leq (d_{X/Y}+i+1)(i+1+ \delta_{ Y/S}+ \chi_2(\Omega_{X/S})).
  \end{displaymath}
In particular, if  $X/S$ is d.c.i., then 
  \begin{displaymath}
    \codim_X^+B_\pi^{(i)}\leq (d_{X/Y}+i+1)(i+1+ \delta_{ Y/S}).
  \end{displaymath}
\item Assume that $X/S$ and $Y/S$ are d.c.i., then
      \begin{displaymath}
        \codim^+_X B^{(i)}_\pi\leq ( d_{X/Y}+i+1)(i+1 +\chi_2(\Omega_{X/Y})) \leq ( d_{X/Y}+i+1)(i+1 +\beta_1(\Omega_{X/Y})).
      \end{displaymath}
\item Assume that $X/S$ is smooth and that each restriction to fibres
  $X_s \to Y_s$, $s\in S$, is generically smooth, then
      \begin{displaymath}
        \codim_X^+ B_\pi \leq \delta_{X/Y} + d_{X/Y} -1.
      \end{displaymath}
    \end{enumerate}
 \end{thma}
 Clearly, (1) is more interesting than (2) and (3) if we have more
 knowledge of the ramification of $X/S$ and $Y/S$ than the
 ramification of $X/Y$.


 Let us consider the case $d_{X/Y}=0$. On the one hand, if
 $\pi: X\to Y$ is a {\it finite } morphism we have classical result by
 Zariski, Nagata and Auslander
 \cites{zariski:purity,nagata:purity,auslander}, stating that $B_\pi$ is
 pure when $\pi$ is finite, $X$ is normal, and $Y$ is regular. This was
 generalised by Grothendieck, Faltings, Griffith, Cutkosky, and
 Kantorovitz
 \cites{SGA2,algebraisation:faltings,griffith:ram,cutkosky,kantorovitz},
 allowing some singularities in $Y$. On the other hand, for birational
 morphisms van der Waerden's purity theorem states that $B_\pi$ is pure
 when $Y$ is normal and a certain condition $(\Wb)$ is satisfied (e.g.
 $Y$ is $\Qb$-factorial). It is therefore natural to ask as in \cite
 [Rem. 21.12.14, (v)]{EGA4:4} if the two types of purity, one for
 finite and another for birational morphisms, can be used together so
 that one gets purity for generically finite morphisms from the mere
 existence of a factorisation into a birational and a finite morphism
 $X \xrightarrow{f}Y'\xrightarrow{g}Y$, where $Y'$ is the
 normalisation of $Y$ in $X$. But there seems no be no reason that
 $Y'$ satisfies $(\Wb)$; more precisely, the complement of the branch
 locus $B_g$ of the finite morphism needs to be affine (see
 \Remark{\ref{bir-fin}}). However, a theorem of Gabber
 \Th{\ref{gabber}} proves purity for such maps $X/Y$ when $Y$ is
 regular. We work out necessary conditions in
 \Theorem{\ref{example-th}} to ensure that this phenomenon does not
 occur also when $Y$ need not be regular. Moreover, we get other quite
 general positive results when $\pi$ is only generically separably
 algebraic, which is thus less than what is required both in the
 Zariski-Nagata-Auslander purity theorem and van der Waerden's purity
 theorem. The first is (1) in \Theorem{\ref{purity-br-loc}}, implying
 $\codim_X B_\pi\leq 1$ when $X/S$ is d.c.i. and $Y/S$ is smooth. However,
 it is often important to allow that either $X/S$ or $Y/S$ not be
 d.c.i.. First in arbitrary characteristic, \Theorem{\ref{simult}}
 gives $\codim_X B_\pi \leq 1$, still assuming that $X/S$ be d.c.i.,
 together with a regularity condition at points of low height in $X$
 and $Y$, respectively, and that the canonical map
 $\pi^*(\Omega_{Y/S})\to \Omega_{X/S}$ be injective. When the base scheme
 $S$ is defined over the rational numbers we have
 \Theorem{\ref{zariski-rel}}, again ensuring $\codim_X B_\pi \leq 1$, which
 can be regarded as a generalised relative version of the
 Zariski-Nagata-Auslander purity theorem in the sense that the
 conditions on $X$ and $Y$ are of a similar type, namely that $Y/S$ is
 smooth and $X$ satisfies $(S_2)$, while $\pi$ is only generically
 algebraic and not necessarily finite. In general, even if $\pi$ is
 finite, $B_\pi$ need not be pure of codimension 1 when $Y/S$ is
 non-smooth and the bound on $\codim B_\pi$ will depend on the type of
 singularities. For example, we include a simple argument for
 Cutkosky's bound $\codim_X B_\pi \leq 2$ when $\pi$ is finite, $X$ is
 normal, and $Y$ is a local complete intersection
 \Prop{\ref{cutkosky}}; it is really a direct consequence of
 Grothendieck's purity theorem. Our rather general bounds that arise
 from \Theorems{\ref{purity of critical}}{\ref{purity-br-loc}} in
 terms of defect numbers of $X/S$ and $Y/S$ are interesting to compare
 to a bound by Faltings and Cutkosky in terms of the regularity defect
 of $Y$, where the latter is applicable only when $\pi$ is finite. Not
 only is our bound easier to get and more general in that $\pi$ need not
 be finite, in the finite case it improves the Faltings-Cutkosky bound
 when $\pi$ is defined over some base $S$ and the singularities of $X$
 and $Y$ are fibered over $S$ (see discussion after
 \Proposition{\ref{cutkosky}}).

 \begin{remark}\label{revised}
   ({\it Added after publication}) I was made aware by Aise Johan de
   Jong that Example 4.3 in the published version is nonsensical so it
   is now removed; he also made me aware of an important result by O.
   Gabber \Th{\ref{gabber}}, which is now included for reference.
   Also, Supravat Sarkar detected an error in
   \Corollary{\ref{cor-pure}} caused by faulty application of
   \Theorem{\ref{fittingtheorem}} (2); this is now corrected.
   \Theorem{\ref{purity-br-loc}(3)} is also corrected. I thank de Aise
   Johan de Jong and Supravat Sarkar for their feedback.
 \end{remark}

 \medskip{\it Generalities}: All schemes are assumed to be N\oe
 therian and we conform to the notation in EGA. The height $\hto (x)$
 of a point $x$ in $X$ is the same as the Krull dimension of the local
 ring $\Oc_{X,x}$, and the dimension of $X$ is $\dim X = \sup \{\ \hto
 (x)\ \vert \ x\in X\}$. A point $x$ is a {\it maximal point} in a
 subset $T$ of $X$ if for each point $y$ in the closure of $x$ in $T$
 we have $\hto(x)\leq \hto(y)$, i.e., if $x_1 \in T$ specialises to
 $x$ and $\hto (x_1)\leq \hto(x)$, then $x_1=x$. Denote by $\Max (T)$
 the set of maximal points of $T$, so $\Max(X)$ consists of points of
 height $0$. A property on $X$ is {\it generic} if it holds for all
 points in $\Max (X)$.  {\it We assume that $X$ is generically
   reduced}, so $k_{X,\xi}= \Oc_{X,\xi}$ when $\xi \in \Max (X)$ (so
 if the nilradical of $\Oc_X$ is non-zero, then all its associated
 points are embedded points in $X$).  Put \begin{align*}
   \codim^+_X T &= \sup \{\ \hto (x) \  \vert \ x \in \Max (T) \},\\
   \codim^-_X T &= \inf \{\ \hto (x) \ \vert \ x \in \Max (T)\},
 \end{align*}
 so $\codim^-_{X} T \leq \hto (x)\leq \codim^+_X T$ when $x\in \Max
 (T)$ (in the introduction we mean $\codim = \codim^+$); one may call
 $ \codim^+_X T$ and $\codim^-_X T$ the upper and lower codimension of
 $T$ in $X$, respectively. To make our formulas hold when $T$ is the
 empty set, put $\codim_X^{+}\emptyset =-1$, since we are interested
 in  upper bounds of $\codim_X ^+ T$; similarly, we are interested in
 lower bounds on $\codim^-_X T$, so we put $\codim_X^-
 \emptyset=\infty$.  For a coherent $\Oc_X$-module $M$ we put
 $\depth_T M = \inf \{\depth M_x \ \vert \ x\in T\}$.  We define the
 {\it relative dimension} of a morphism locally of finite type $\pi:
 X\to Y$ at a point $x\in X$ as the infimum of the dimension of the
 vector space of K{\"a}hler differentials at all maximal points $\xi$
 that specialise to $x$, i.e.
 \begin{displaymath}
   d_{X/Y,x}=\inf \{ \dim_{k_{X,\xi}}\Omega_{X/Y, \xi} \ \vert \ x\in \xi^-, \xi \in \Max (X)\}.
 \end{displaymath}
 Put $d_{X/Y}= \inf \{ \dim_{k_{X,\xi}}\Omega_{X/Y, \xi} \vert \
 \xi\in \Max (X)\} $.  If $X/Y$ is generically smooth $X$, and the
 numbers $d_{X/S, \xi}$ and $d_{Y/S, \pi(\xi)}$ do not depend on the
 choice of maximal point $\xi \in \Max (X)$ that specialises to $x$,
 then $d_{X/Y,x}= d_{X/S, x}- d_{Y/S,\pi(x)}$.  Recall that
 $\dim_{k_{X,\xi}}\Omega_{X/Y, \xi}=
 \dim_{k_{X,\xi}}\Omega_{k_\xi/k_{\pi(\xi)}} $ is the same as the
 transcendence degree and dimension of a $p$-basis of
 $k_{\xi}/k_{\pi(\xi)}$ in characteristic $0$ and $p$, respectively,
 and these numbers are equal when the field extension is separable and
 finitely generated.  Thus $d_{X/Y,x}=0$ when $x\in \xi^-$, $\xi \in
 \Max (X)$, $\Oc_{X, \xi}= k_{X,\xi}$ and $k_{X, \xi}/k_{\pi(\xi)}$ is
 separably algebraic, and if moreover $X$ is integral, then $d_{X/Y,x}
 = \trdeg k_{X,\xi}/k_{Y,\pi(\xi)}$.

 \section{Critical scheme and branch scheme }\label{pure-critical}
 Assume that $X$ is a connected scheme and $\pi: X\to Y$ be a dominant
 morphism, which is locally of finite type.  The first fundamental
 exact sequence of quasi-coherent $\Oc_X$-modules
  \begin{equation}\label{eq:can-ex}
  0 \to \Gamma_{X/Y/S} \to     \pi^*(\Omega_{Y/S}) \xrightarrow{p}
  \Omega_{X/S} \to \Omega_{X/Y} \to 0.
\end{equation} 
contains the kernel $\Gamma_{X/Y/S}$ of $p$, which is called the
imperfection module of $X/Y/S$.  Denoting the image by $\Vc_{X/Y/S}$
we have two short exact sequences
\begin{eqnarray}\label{eq:can-ex1}
  0\to \Gamma_{X/Y/S}\to \pi^*(\Omega_{Y/S}) \to \Vc_{X/Y/S}\to 0, \\
  0 \to \Vc_{X/Y/S}\to \Omega_{X/S}\to \Omega_{X/Y}\to 0.\label{eq:can-ex2}
\end{eqnarray}

Consider a chain of morphisms $X \xrightarrow{i} X_r \xrightarrow{p} Y
\to S$ and put $\pi= p\circ i$. There is an exact sequence \cite[Th.
20.6.17]{EGA4:1}
\begin{equation}
  \label{imperfection-exact}
  0 \to \Gamma^X_{X_r/Y/S}\to \Gamma_{X/Y/S} \to \Gamma_{X/X_r/S} \to i^*(\Omega_{X_r/Y})\to \Omega_{X/Y} \to \Omega_{X/X_r}\to 0,    
\end{equation}
where
\begin{displaymath}
  \Gamma^X_{X_r/Y/S} = \Ker (i^* (p^*(\Omega_{Y/S})))\to i^*(\Omega_{X_r/S})).
\end{displaymath}
We will study the Fitting ideals $F_{d_{X/Y}+i}(\Omega_{X/Y})$, $i\geq
0$, defining the $i$th branch scheme $B^{(i)}_\pi$ (K\"ahler
different, Jacobians), so there is a finite decreasing filtration of
$B_\pi$
\begin{displaymath}
\cdots  \subseteq B_\pi^{(i)}\subseteq B_\pi^{(i-1)}\subseteq  \cdots \subseteq B_\pi^{0} = B_\pi.
\end{displaymath}
Here $B_\pi$ is the {\it branch scheme}\/ and its underlying
topological space is the {\it branch locus}.
\begin{remark}
  \begin{enumerate}
  \item If $Y/S$ is smooth and $X/Y$ is generically smooth, then
    $\Gamma_{X/Y/S}=0$, but $\Gamma_{X/Y/S}$ is normally non-zero when
    $Y/S$ is non-smooth, also in characteristic $0$.  Example: $A=
    k[x,y]/(x^2+y^3), B= k[x',y']/(x'^2+y')$ and let $ A \to B$ be
    defined by $x'=xy$ and $y'=y$ (a chart of the strict transform
    with respect to the blow-up of the origin). The torsion submodule
    of $\Omega_{A/k}$ is $A(2ydx -3xdy)$ and $\Gamma_{B/A/k} = B (
    2ydx -3xdy )\subset B\otimes_A \Omega_{A/k}$.

    \item The two middle terms in (\ref{eq:can-ex1}) are
      quasi-coherent so $\Gamma_{X/Y/S}$ is quasi-coherent, and if
      $Y/S$ is of finite type, then $\Gamma_{X/Y/S}$ is coherent,
      since $X$ is a noetherian space. Also, if only $X/Y$ is locally
      of finite type, then $\Gamma_{X/Y/S}$ is coherent; this is
      proven using the sequence (\ref{imperfection-exact}).  Assume
      that $X,Y,S$ are integral and that the fraction field of $S$ is
      perfect. If $\gamma_{X/Y/S}$ is the generic rank of
      $\Gamma_{X/Y/S}$, then $d_{X/Y}= \trdeg
      k_{X,\xi}/k_{Y,\pi(\xi)}+ \gamma_{X/Y/S}$ (Cartier's equality),
      and if $X/Y$ is generically smooth, then $d_{X/Y}=\trdeg
      k_{X,\xi}/k_{Y,\pi(\xi)}$ (see \cite[\S 20.6]{EGA4:1}).
    \item We have $B_\pi^{(i)}=X$ when $i < 0$ and $B_\pi^{(i)}=
      \emptyset$ when $i \geq \sup_x \{\beta_0(\Omega_{X/Y,x})\}
      -d_{X/Y}$.  By our definition of $d_{X/Y}$ it follows that
      $\codim_X^{-} B_\pi \leq 1$, always, and if $X/Y$ is generically
      smooth, then $B_\pi = B_\pi^{(0)}$. If $d_{X/Y}=0$, then $B_\pi
      = \supp \Omega_{X/Y}$, and $B_\pi$ is the locus of points where
      $X/Y$ is not formally unramified, while if $d_{X/Y}>0$, then
      $B_\pi$ is a subset of the non-smoothness locus of $X/Y$ ; if
      $\pi$ is locally of finite type and flat, then $B_\pi$ is the
      non-smoothness locus.
  \end{enumerate}
\end{remark}

The dual of $p$ induces a homomorphism of $\Oc_Y$-modules, the {\it
  tangent morphism}\/ of $\pi$, $ d\pi:
T_{X/S}=Hom_{\Oc_X}(\Omega_{X/S}, \Oc_X) \to T_{X/S\to Y/S}$ where the
$\Oc_X$-module of `derivations from $\Oc_Y$ to $\Oc_X$' is
\begin{displaymath}
  T_{X/S\to Y/S}   =Hom_{\Oc_X}(\pi^*(\Omega_{Y/S}),\Oc_X)=Hom_{\pi ^{-1}(\Oc_Y)}(\pi
  ^{-1}(\Omega_{Y/S}),\Oc_X),
\end{displaymath}  
and is part of the exact sequence
\begin{equation}\label{fund-exact} 0 \to T_{X/Y} \to T_{X/S}
  \xrightarrow{d\pi }
  T_{X/S \to Y/S} \to \Cc_{X/Y} \to 0,
\end{equation} where $\Cc_{X/Y}$ is the {\it critical
  module}\/ of $\pi$. 
The critical set $C_\pi=\supp \Cc_{X/Y}$  is endowed with a structure of a scheme (still denoted $C_\pi$), defined by  the Fitting ideal $F_0(\Cc_{X/Y})$; we say that $\pi$ is 
submersive at a  point $x$ in $X$ if $x\notin C_\pi$. Note that 
$T_{X/S\to Y/S} = \pi^*(T_{Y/S})$ when either $\pi$ is flat or $\Omega_{Y/S}$ is locally 
free of finite rank. Let  $C_{\pi}^{(i)}$ be the scheme that  is defined by
the Fitting ideal $F_i(\Cc_{\pi})$, giving a finite decreasing filtration  of the critical scheme $C_\pi$
\begin{displaymath}
  \subset C_\pi^{(i)}\subset   \cdots \subset C_\pi^{(1)}\subset   C_\pi^{(0)} = C_\pi.
\end{displaymath}

\begin{remark}
  If $\Omega_{X/S}$ is locally free it is straightforward to see that
  the space of $B^{(i)}_\pi$ is given by a rank condition on the
  induced map of fibres of the map $p$, while if $T_{Y/S}$ is locally
  of finite rank, the space of $C^{(i)}_\pi$ is given by a rank
  condition on the induced map of fibres of the map $d\pi$.
\end{remark}

\begin{proposition}\label{dual-crit} If $X/S$ and $Y/S$ are
  generically smooth and $B_{X/S}= \emptyset$, $B_{Y/S}= \emptyset$,
  then
  \begin{displaymath}
    F_i(\Cc_{X/Y}) = F_{d_{X/Y}+i} (\Omega_{X/Y}).
  \end{displaymath}
\end{proposition}
To be clear, note that the requirements in
\Proposition{\ref{dual-crit}} are satisfied when $X/S$ and $Y/S$ are
smooth, and thus in this situation $C_\pi^{(i)} = B_\pi^{(i)}$.  The
reason is that $\Cc_{X/Y}$ and $\Omega_{X/Y}$ are transposed modules
of one another, so for the proof one needs a relation between the
Fitting ideals of a module and its transpose.
\begin{lemma}\label{dual-fitting} Let $\phi : G_1 \to G_2$ be a
  homomorphism of locally free $\Oc_X$- modules of finite ranks $g_1$
  and $g_2$, respectively. Let $\phi^* : G_2^* \to G_1^* $ be the
  homomorphism of dual modules.  Then
  \begin{displaymath} F_i(\Coker \phi) = F_{g_1-g_2 +i}(\Coker
    \phi^*).
  \end{displaymath}
\end{lemma}
\begin{proof} $F_i(\Coker \phi)$ is the image of the map
  $\land^{g_2-i}G_1 \otimes_{\Oc_Y}(\land^{g_2-i}G_2)^* \to \Oc_X$
  induced by the map $\land^{g_2-i}\phi : \land ^{g_2-i}G_1 \to
  \land^{g_2-i}G_2$ and $F_r(\Coker \phi^*)$ is the image of the map
  $\land^{g_1-r}G_2^* \otimes_{\Oc_Y}(\land^{g_1-r}G_1^*)^* \to \Oc_X$
  induced by the map $\land^{g_1-r}\phi^* : \land ^{g_1-r}G_2^* \to
  \land^{g_1-r}G_1^*$. When $g_2-i=g_1-r$, i.e.  $r= g_1-g_2+i$ we get
  a commutative diagram
  \begin{displaymath} \xymatrix{\land^{g_2-i}G_1
      \otimes_{\Oc_Y}(\land^{g_2-i}G_2)^* \ar[r]\ar[d]&
      \Oc_X\ar@{=}[d]\\ (\land^{g_2-i}G_1^*)^*
      \otimes_{\Oc_Y}\land^{g_2-i}G_2^* \ar[r]& \Oc_X }
  \end{displaymath} where the left vertical homomorphism
  exists because there are canonical maps $\land^{g_2-i}G_1 \to
  (\land^{g_2-i}G_1^*)^*$ and $\land^{g_2-i}G_2^* \to
  (\land^{g_2-i}G_2)^*$, and the latter is an isomorphism because
  $G_2$ is locally free (they both are isomorphisms since $G_1$
  also is locally free).
\end{proof}

\begin{pfof}{\Proposition{\ref{dual-crit}}} The assumptions give that
  $\Omega_{X/S}$ and $\Omega_{Y/S}$ are locally free of ranks
  $d_{X/S}$ and $d_{Y/S}$, respectively, so $d \pi: G_1=T_{X/S}\to
  G_2= \pi^*(T_{Y/S})$ is a homomorphism of locally free
  $\Oc_X$-modules, where $\Coker (d \pi)= \Cc_{X/Y}$ and $\Coker (d
  \pi ^*)= \Omega_{X/Y}$, so the result follows from
  \Lemma{\ref{dual-fitting}}, noting that $d_{X/Y}= d_{X/S}-d_{Y/S}$.
\end{pfof}

\begin{remark}\label{remark:transpose}

  Recall that the kernel and cokernel of a biduality morphism $M\to
  M^{**}$ of a coherent $\Oc_X$-module $M$ can be expressed using the
  transposed module $D(M)$, locally defined up to local projective
  equivalence by $D(M)= \Coker (\phi^*)$, where $\phi$ is a local
  presentation $F_1 \xrightarrow{\phi} F_0 \to M\to 0$.  Then we have
  the exact sequence
  \begin{equation}\label{auslander-bridger}
    0 \to Ext^1_{\Oc_X}(D(M), \Oc_X)\to M \to M^{**}\to Ext^2_{\Oc_X}(D(M), \Oc_X)\to 0;
  \end{equation}
  see \cite{auslander-bridger}.  Note also that when the projective
  dimension $\pdo M_x \leq 1$ for each point $x$, then $D(M)$ is
  locally projectively equivalent to $Ext^1_{\Oc_X}(M,$ $ \Oc_X)$.
\end{remark}

\section{Differentially complete
  intersections}\label{diffcompletesection} 
Ferrand and Vasconcelos \cites{ferrand:suite,vasconcelos:normality}
have shown that if $X/k$ is a reduced scheme locally of finite type
and generically smooth (i.e. the residue fields at all maximal points
are separable over $k$), then $X/k$ is a locally complete intersection
if and only if the projective dimension $\pdo \Omega_{X/k,x}\leq 1$ at
each point $x$; see also \cite[\S9]{kunz:kahler}. Because of this
result there are two natural notions of ``locally complete
intersection morphisms''. The first is well-known in the case when
$\pi: X\to Y$ is ``smoothable'': there exists a locally defined
factorisation $X\to Z \to Y$, where $X/Z$ is a regular immersion and
$Z/Y$ is formally smooth, i.e.  the ideal of $X$ in $Z$ is locally
defined by a regular sequence; this was further developed in
\cite{avramov:lci} to general morphisms employing
``Cohen-factorisations'', proving that there is an alternative
definition by the vanishing of certain Andr\'e-Quillen homology
groups.  We continue to call such morphisms locally complete
intersection morphisms (l.c.i.), but we shall however have more use
for a second possibility.  Say that a dominant morphism $\pi : X\to Y$
is a {\it differential complete intersection} (d.c.i.) at a point $x$
if $\pdo \Omega_{X/Y,x}\leq 1$, and that $\pi$ is a d.c.i.  if it is
d.c.i.  at each point $x$; we then also write $\pdo \Omega_{X/Y}\leq
1$.  Let $x $ be a specialisation of a point $\xi$ in $X$.  Since
$\Omega_{X/Y,\xi}= \Oc_{X,\xi}\otimes_{\Oc_{X,x}}\Omega_{X/Y,x}$, it
is evident that a morphism is a d.c.i. at $\xi$ if it is d.c.i.  at
$x$, hence it suffices to check the closed points in $X$ to see if a
morphism is d.c.i. If the first syzygy of the quasi-coherent module
$\Omega_{X/Y}$ is coherent it is clear that the set $\{x\in X \ \vert
\ \pi \text{ is a d.c.i.  at } x\}$ is open. Recall also that
$\Omega_{X/Y,x}= \Omega_{\Oc_{X,x}/\Oc_{Y,\pi(x)}}$ (see
\cite{kallstrom:preserve} for a proof not using the fact that
$\Oc_{X,x}\to \Oc_{X,\xi}$ is etale), so d.c.i. is a property of the
morphism of local rings $\Oc_{Y,\pi(x)}\to \Oc_{X,x}$.  If $\pi$ is
smoothable and l.c.i. at $x$ then it is l.c.i at $\xi$, but the proof
of this assertion is not as immediate as for the d.c.i.  property; for
non-smoothable morphisms this localisation property for
l.c.i. morphisms need not hold, see \cite[5.3, 5.12]{avramov:lci}.

\begin{theorem}(Ferrand, Vasconcelos)\label{ferrand-vasconcelos} Let
  $\pi : X/S\to Y/S $ be morphism which is locally of finite type, and
  consider the following properties  of a point $x$ in $X$:
  \begin{enumerate}[label=(\theenumi),ref=(\theenumi)]
  \item $\pi$ is l.c.i. at $x$.
  \item $\pi$ is d.c.i. at $x$.
  \end{enumerate}
  If $X/Y$ is smooth at all associated points in $X$, then $(1)\Rightarrow
  (2)$. If $X/Y$ is generically smooth and $\pdo_{\Oc_{Y,\pi(x)}}
  \Oc_{X,x} <\infty$,  then $(2)\Rightarrow (1)$.
\end{theorem}
\begin{remark} Note that we do not require that our l.c.i. or d.c.i
  morphsisms be flat; however, if $\pi$ is l.c.i. then $\pi$ will have
  a finite flat dimension. Kunz \cite[Th. 9.2]{kunz:kahler} gives a
  part of the proof in the above relative situation, assuming $\pi$ is
  flat, but it seems to me as if the possibility of embedded points
  was overlooked.
\end{remark}
We record a situation where no embedded points are present in $X$, so
the above smoothness conditions  at associated points can be expressed more concretely as
$X$ being geometrically reduced over the maximal points in $Y$. The
proof is immediate.
\begin{lemma}
  If $Y$ is Cohen-Macaulay and $X/Y$ is l.c.i., then $X$ is
  Cohen-Macaulay, and hence contains no embedded points.
\end{lemma}

The proof of the following lemma in
\cite{vasconcelos:normality} is perhaps a little succinct, so we
include an argument.
\begin{lemma}(Vasconcelos)\label{vasconcelos-lemma}
Let $J$ be a proper ideal of a N{\oe}therian local ring $A$, such that $\pdo _A J < \infty$. If $J/J^2= \Gamma \oplus K$ where $\Gamma$ is free of rank $l$ over $A/J$, then $J$ contains a regular sequence of length $l$, and if $K\neq 0$ it contains a regular sequence of length $l+1$.
  \end{lemma}
  \begin{proof} Put $B=A/J$. Since $\pdo_A J < \infty $, hence $\pdo_A
    B< \infty $, and since $A$ is local, $B$ has a finite free
    resolution as $A$-module.  Since $\Ann_A (B)= J\neq 0$, by
    Auslander-Buchsbaum's theorem \cite{auslander-buchsbaum} $J$
    contains an $A$-regular element, so $J \not \subset P$ for each
    associated prime $P$ of $A$. By prime avoidance there exists an
    element $x\in J$ such that $x\not \in P$ for all associated primes
    (so again $x$ is a regular element) and $x\notin \mf_A J$. The
    image $\bar x$ of $x$ for the projection $J/J^2 \to \Gamma$
    satisfies $\bar x \not\in \mf_B \Gamma $, so it can be
    complemented to a basis $\{\bar x_1= \bar x, \bar x_2, \dots ,
    \bar x_l \}$ of the free $B$-module $\Gamma$, hence $B \bar x $ is
    a free summand of $\Gamma\subset J/J^2$. Select $x_i \in J$ that
    project to $\bar x_i$, $i=2, \dots , l$.

    Now put $A^* =A/(x)$ and $J^*=J/(x)$. Since $x$ is $A$-regular it
    is also $J$-regular, so by \cite[Lem. 2, \S 18]{matsumura}
    $\pdo_{A^*}J/xJ < \infty$. Since $x\not\in \mf_A J$ it follows
    that the natural map $J/xJ \to J^*$ splits (see proof of [Th 19.2,
    loc cit]). Therefore $\pdo_{A^*} J^* \leq \pdo _{A^*}J/xJ <
    \infty $. Since $B\bar x$ is a free summand of $\Gamma$ and hence
    a free direct summand of $J/J^2 $, it follows that $J^*/(J^*)^2=
    J/(Ax + J^2) \cong \Gamma^* \oplus K^* $, where $\Gamma^*$ is a
    free module of rank $l-1$, generated by $x_i \omod ((x)+ J^2)$,
    $i=2, \dots , l$.  If $K=0$ we see by induction that $I$ is
    generated by a regular sequence of length $l$.  If $K\neq 0$ again
    by induction it follows that $I$ contains a regular sequence of
    length $l+1$.
  \end{proof}
\begin{pfof}{\Theorem{\ref{ferrand-vasconcelos}}}
  $(1)\Rightarrow (2)$: There exists locally a factorisation $X
  \xrightarrow{i} X_r \to Y$ where $X/X_r$ is a regular immersion and
  $X_r/Y$ is smooth. Letting $I$ be the ideal of $X$ in $X_r$ we get
the exact sequence
  \begin{displaymath}
0\to  K \to     I/I^2 \to i^*(\Omega_{X_r/Y})\to \Omega_{X/Y}\to 0.
  \end{displaymath}
  Thus $I$ is generated by a regular sequence at each point $x$, so
  the $\Oc_{X,x}$-module $I_x/I_x^2$ is free, and $\Omega_{X_r/Y,x}$
  is free, so it follows that $\pdo \Omega_{X/Y,x}\leq 1$ if we prove
  that $K=0$. Since $I/I^2$ is locally free this will follow if
  $K_x=0$ when $x$ is an associated point. By assumption
  $\Omega_{X/Y,x}$ is free of rank $d_{X/Y}$, hence $I_x/I_x^2
  =\Gamma\oplus K_x $, where $\Gamma$ is free of rank $l:= d_{X_r/Y}
  -d_{X/Y}= \codim_{X_r} X$, since $X/X_r$ is a regular immersion.
  Since $X= V(I)$ it follows that $I_x$ does not contain a regular
  sequence of length $l+1$, hence by \Lemma{\ref{vasconcelos-lemma}}
  $K_x=0$. Since $K$ is a submodule of a locally free $\Oc_X$-module
  which is $0$ at all the associated points of $X$ it follows that
  $K=0$.

  $(2)\Rightarrow (1)$: There exists locally a factorisation
  $X\xrightarrow{i} X_r \to Y$ where $X_r/Y$ is smooth and $X/X_r$ is
  a closed immersion.  Consider the exact sequence
  \begin{displaymath}
    0 \to     \Gamma_{X/X_r/Y}\to i^*(\Omega_{X_r/Y})\to \Omega_{X/Y}\to 0.
  \end{displaymath}
  Since $\pdo \Omega_{X/Y,x}\leq 1$ at each point $x$ and
  $\Omega_{X_r/Y,x}$ is free, it follows that $ \Gamma_{X/X_r/Y,x}$ is
  free.  Putting $l= d_{X_r/Y,x} - d_{X/Y,x}$, since $X/Y$ and $X_r/Y$
  are generically smooth, we have $\rank \Gamma_{X/X_r/Y}= l$.
  Combining with the exact sequence in the beginning of the proof we
  get the split exact sequence
  \begin{displaymath}
    0 \to K \to     I_x/I_x^2 \to \Gamma_{X/X_r/Y,x}\to 0,
  \end{displaymath}
  where $K$ is torsion.  Put $A = \Oc_{X_r,x}$, $J= I_x$ and $\Gamma =
  \Gamma_{X/X_r/Y,x}$, so $J/J^2 = \Gamma \oplus K$, where $\Gamma$ is
  $A$-free of rank $l$, and as the maximal length of an $A$- regular
  sequence in $J$ satisfies $\depth_J A \leq \dim A -\dim \Oc_{X,x}=l$,
  \Lemma{\ref{vasconcelos-lemma}} implies $K=0$; we only have to note
  that $\pdo_{\Oc_{Y,\pi(y)}}\Oc_{X,x} < \infty$ implies
  $\pdo_{A}\Oc_{X,x}< \infty$, and therefore $\pdo_{A} J < \infty$.
\end{pfof}
\begin{remark}
  \begin{enumerate}\item 
    If $\pi$ is not generically smooth, then $(1)$ does not imply (2)
    in \Theorem{\ref{ferrand-vasconcelos}}. Example: $A=k[x]/(x^2)$ is
    l.c.i. over $k$, but $\pdo_A \Omega_{A/k} = \pdo_A k = \infty$.
  \item For a regular base $Y$, generically smooth d.c.i. morphisms
    are the same as generically smooth locally complete intersection
    morphisms, while  if $\pdo_{\Oc_{Y,\pi(x)}} \Oc_{X,x} = \infty$ for
    some point $x$, then (2) does not imply (1) in
    \Theorem{\ref{ferrand-vasconcelos}}. Example: $A=k[t]/(t^2)$, $k$
    is a field, $R=A[x,y]$ and $I=(t+x,t+yx^2)$. Put $B=R/I$ and
    consider the natural map $A\to B$. Then
    \begin{enumerate}
    \item In the ring $R$ we have $txy(t+x)= t(x^2y+t)$ and $I$ cannot
      be generated by a regular sequence.
    \item The $B$-module $I/I^2$ is free of rank $2$, so by (a) and
      \cite{vasconcelos:reg-seq} $\pdo_R I =\infty$.
    \item The sequence $0\to I/I^2 \to B\otimes_R \Omega_{R/A} \to
      \Omega_{B/A} =\frac{k[t,x,y]}{(t^2,x^2,y)}\to 0$ is exact also
      to the left.
    \end{enumerate}
    Therefore $\pdo \Omega_{B/A}\leq 1$ while $A\to B$ is not l.c.i.
    See also \Theorem{\ref{comp-dci}}, (3-4).
  \end{enumerate}

\end{remark}

\begin{lemma}\label{drsci-smooth} 
  Let $\pi: X\to Y$ be a locally of finite type morphism that is
  smooth at all associated points in $X$. Assume either:
  \begin{enumerate}[label=(\theenumi),ref=(\theenumi)]
  \item \label{dolg-lemma} $X/S$ and $Y/S$ are smooth.
  \item\label{reg-complete} $X$ and $Y$ are regular schemes.
  \end{enumerate} Then $\pi$ is d.c.i.
\end{lemma}
\ref{dolg-lemma} was first observed by Dolgachev
\cite{dolgachev:nonsmooth}; in \Theorem{\ref{comp-dci}} we will give
another necessary condition for $X/Y$ to be d.c.i.  Clearly,
$(1)\Rightarrow (2)$ when $X$ and $Y$ are geometrically regular over a
field.
\begin{proof}
  (1): Since $\pi^*(\Omega_{Y/S})$ is locally free and $\pi$ is smooth
  all associated points, implying that the imperfection module
  $\Gamma_{X/Y/S}$ is $0$ at all associated points, hence
  $\Gamma_{X/Y/S}=0$ in the exact sequence (\ref{eq:can-ex}). Since
  moreover $\Omega_{X/S}$ is locally free, it follows that $\pdo
  \Omega_{X/Y,x}\leq 1$ at each point $x$.  (2): There exists locally
  a factorisation $X \xrightarrow{i} X_r \xrightarrow{p} Y$, where $i$
  is a closed immersion and $p$ is smooth.  Since $p$ is smooth and
  $Y$ is regular, it follows that $X_r$ is regular. Since $X$ is
  regular and $i$ is a closed immersion, it must be a regular
  immersion. Therefore $\pi$ is l.c.i., hence by
  \Theorem{\ref{ferrand-vasconcelos}} $\pi$ is d.c.i.
\end{proof}
We give necessary conditions to conclude that a morphism is d.c.i.
when all its fibres are d.c.i.  

\begin{proposition}\label{drci-flat} Let $\pi: X\to Y$ be a flat
  dominant morphism locally of finite type of N\oe therian schemes,
  which is smooth at all associated points of $X$.  Assume either of
  the conditions:
  \begin{enumerate}[label=(\theenumi),ref=(\theenumi)]
  \item $\Omega_{X/Y}$ is $Y$-flat.
  \item each fibre $X_y/k_{Y,y}, y\in Y$, is generically smooth
    (i.e. generically geometrically reduced).
  \end{enumerate}
  If the fibre $X_y/k_{Y,y}$ is d.c.i., then $\pi$ is d.c.i. at each
  point $x$ in $X_y \subset X$. Hence if each (closed) fibre $X_y$ is
  d.c.i., then $\pi$ is d.c.i.
\end{proposition}
Andr\'e \cite{andre:pseudoregular} has studied flat morphisms such
that all fibres $X_y/k_{Y,y}$ are geometrically reduced and l.c.i.,
characterising them in terms of the cotangent complex.

For the proof of \Proposition{\ref{drci-flat}} we need a standard lemma.
\begin{lemma}\label{flatlemma} Let $ \pi : X\to Y$ be a flat morphism
  of schemes and $M$ be a coherent $\Oc_X$-module, flat over $Y$. Let
  $x$ be a point in $X$, $X_y$ be the fibre over $y= \pi(x)$, and
  $M_{X_y}$ the restriction of $M$ to $X_y$. If $\pdo M_{X_y,x}\leq
  1$, then $\pdo M_x \leq 1$.
\end{lemma}
\begin{proof} Locally there exists a presentation $0 \to L \to F \to M
  \to 0$ where $F$ is locally free of finite rank.  It suffices to see
  that $L_x$ is free.  Applying $k_{Y,y}\otimes_{\Oc_{Y,y}}\cdot $ to
  the exact sequence, by assumption $k_{Y,y}\otimes_{\Oc_{Y,y}}L_{x}$
  is free over $\Oc_{X_y,x}$, since $M_x$ is flat over $\Oc_{Y,y}$ and
  $\pdo_{\Oc_{X_y,x}} k_{Y,y}\otimes_{\Oc_{Y,y}}M_x \leq 1$. Selecting
  a basis $k_{Y,y}\otimes_{\Oc_{Y,y}}\Oc_{X,x}^n \cong
  k_{Y,y}\otimes_{\Oc_{Y,y}}L_{x}$, arising from a homomorphism $u:
  \Oc_{X,x}^n \to L_x$ of $\Oc_{X,x}$-modules, so $u$ is surjective by
  Nakayama's lemma.  Since $M_x$ and $\Oc_{X,x}$ are flat, hence $L_x$
  is flat over $\Oc_{Y,y}$, we conclude that $u$ is an isomorphism
  \cite[Th.  22.5]{matsumura}.
\end{proof}
\begin{pfof}{\Proposition{\ref{drci-flat}}} (1): This follows
  immediately from \Lemma{\ref{flatlemma}}, noting that
  $k_{Y,y}\otimes_{\Oc_{Y,y}}\Omega_{X/Y,x} = \Omega_{X_y/k_{Y,y}}$.

  (2): We can assume that $X/Y$ is a subscheme of a smooth scheme
  $X^r/Y$ so there is the short exact sequence
  \begin{displaymath}
    0 \to \Lambda  \to \Oc_X \otimes_{\Oc_{X^r}}\Omega_{X^r/
      Y}\to \Omega_{X/Y}\to 0,
  \end{displaymath}
  and the assertion is that $\Lambda_x$ is free over $\Oc_{X,x}$ when
  $x\in \pi^ {-1}(y)$. Since $X^r/Y$ is smooth it follows that the two
  terms to the right are coherent, so $\Lambda_x$ is of finite type.
  Let $I$ be the defining ideal of $X$ in $X_r$ (defined locally). We
  have a surjective map $I_x/I_x^2 \xrightarrow{d_x} \Lambda_x\to 0$.
  Since $\pdo (\Omega_{X_y/k_{Y,y},x}) \leq 1$ for each point $x\in
  X_y$ if follows by \Theorem{\ref{ferrand-vasconcelos}} that
  $X_y/k_{Y,y}$ is l.c.i., since $X_y/k_{Y,y}$ is generically smooth;
  hence since $\pi$ is flat, $I_x$ is generated by a regular sequence
  so $I_x/I_x^2$ is free over $\Oc_{X,x}$.  Since $X/Y$ is smooth at
  each associated point $\xi$ in $X$, and therefore $d_\xi$ is
  injective, it follows that $d_x$ is injective; hence $\Lambda_x =
  I_x/I_x^2$ is free.  Therefore $\pdo \Omega_{X/Y,x}\leq 1$.
\end{pfof}

Consider a base change diagram over
some scheme $S$:
 \begin{displaymath}\tag{$BC $}
    \xymatrix{
 X_1 \ar[r] ^j\ar[d]^{\pi_1}  
& X\ar[d]^\pi\\
      Y_1 \ar[r] & Y,
    }
  \end{displaymath}
  where $X_1=X\times_{Y}Y_1$. The class of d.c.i. morphisms behaves
  almost as well under base-change as the class of smooth morphisms.
\begin{theorem}\label{base-change:dci}
  Let $\pi: X/S \to Y/S$ be a dominant morphism which is locally of
  finite type, where $X/S$ and $Y/S$ are N\oe therian, and assume that
  $\pi$ is smooth at all points in $j(X_1)\subset X$ that are images
  of associated points in $X_1$. If $\pi$ is d.c.i., then $\pi_1:X_1\to
  Y_1$ is d.c.i.
\end{theorem}
Note that when $\pi$ is flat, then $\pi$ is smooth at the associated
points of $j(X_1)\subset X$ if and only if $\pi_1$ is smooth at the
associated points of $X_1$ \cite[Th 19.7.1]{EGA4:1}.

\begin{lemma}\label{diffcomplete-lem}
  Let $j : X\to Y$ be a morphism of schemes and $M$ a coherent
  $\Oc_Y$-module satisfying $\pdo M_y \leq 1$ at each point $y$ in the
  image $j(X)$.  Assume also that $M_y$ is flat over $\Oc_{Y,y}$ when
  $y$ is an associated point in $j(X)$.  Then $ \pdo j^*(M)_x \leq 1$
  at each point $x$ in $X$.
\end{lemma}
\begin{proof} Let $x$ be a point in $X$ and set $B= \Oc_{X,x}$ and $A=
  \Oc_{Y,\pi(x)}$.  If $0\to F^1 \to F^0 \to M_{j(x)}\to 0$ is exact,
  where $F_1, F_0$ are free, then $0 \to Tor^1_A(B, M_{j(x)}) \to
  B\otimes_A F^1 \to B\otimes_A F^0 \xrightarrow{h} j^*(M)_x \to 0$ is
  exact. By assumption the support of tha $B$-module $Tor^1_A(B,
  M_{j(x)})$ does not contain any associated point of $B$ and
  $B\otimes_A F^1$ is  free; therefore $Tor^1_A(B, M_{j(x)})=0$,
  implying $\pdo j^*(M)_x \leq 1$.
\end{proof}
\begin{pfof} {\Theorem{\ref{base-change:dci}}}
  This follows from \Lemma{\ref{diffcomplete-lem}}, noting that
  $\Omega_{X_1/Y_1} = j^*(\Omega_{X/Y})$.
\end{pfof}

\Theorem{\ref{ferrand-vasconcelos}} immediately implies the following
corollary to \Theorem{\ref{base-change:dci}}:
\begin{corollary}\label{gen-avramov} Make the same assumptions as in
  \Theorem{\ref{base-change:dci}}, and assume moreover that
  \begin{displaymath}
  \pdo_{\Oc_{Y_1,\pi(x)}}\Oc_{X_1,x} < \infty
\end{displaymath}
for all $x$ in $X_1$
  (e.g.  $Y_1$ is regular). Then $\pi_1$ is l.c.i..
\end{corollary}
\begin{remark}\label{base-change-dci}
  Assume that $\pi$ in the above diagram is l.c.i.. In
  \cite[5.11]{avramov:lci}, one {\it essentially} requires that either
  $Y_1/Y$ or $X/Y$ to be flat to infer that $\pi_1$ be l.c.i..
  \Corollary{\ref{gen-avramov}} implies that it suffices that $\pi$ be
  smooth at the associated points of $j(X_1) $ when $Y_1$ is regular
  \Th{\ref{ferrand-vasconcelos}}.  It is easy to see that $X_1/Y_1$ is
  d.c.i. when $X/Y$ is d.c.i. and flat.
\end{remark}

\begin{proposition}\label{gamma0} Let $X$ and $Y$ be N\oe
  therian schemes and $\pi : X/S \to Y/S$ be a morphism which is locally of
  finite type, and  smooth at all the  associated points in $X$.  Assume
  either of the following conditions:
  \begin{enumerate}[label=(\theenumi),ref=(\theenumi)]
  \item $\pi$ is l.c.i. (or d.c.i. and
    $\pdo_{\Oc_{Y,\pi(x)}}\Oc_{X,x}< \infty$, $x\in X$)
  \item $X/S$ is d.c.i.
  \end{enumerate}
  Then $\Gamma_{X/Y/S}=0$.
\end{proposition}
\begin{proof}
  There exists locally a factorisation of $\pi$ of the form $X \to X_r
  \to Y$ where $X_r/Y$ is smooth and $X/X_r$ is closed immersion.
  Hence $\Omega_{X/X_r}=0$ and $\Gamma_{X_r/Y/S}=0$, and since
  $\Omega_{X_r/Y}$ is locally free we also get $\Gamma_{X_r/Y/S}^X=0$.
  Therefore the exact sequence (\ref{imperfection-exact}) gives the
  short exact sequence
  \begin{displaymath}
    0 \to \Gamma_{X/Y/S}\to \Gamma_{X/X_r/S}\xrightarrow{\delta } \Gamma_{X/X_r/Y}\to 0.
  \end{displaymath}
  (1): Since $X/Y$ is d.c.i, so $\pdo \Omega_{X/Y,x}\leq 1$ at each
  point $x\in X$, and $i^*(\Omega_{X_r/Y})$ is locally free, it
  follows that $\Gamma_{X/X_r/Y}$ is locally free.  Since $X/Y$ is
  generically smooth it follows that $\Gamma_{X/Y/S}$ is torsion, but
  we want $\Gamma_{X/Y/S}=0$.  Let $I$ be the ideal of $X$ in
  $X_r$. There exist surjections $ I/I^2\xrightarrow{\alpha}
  \Gamma_{X/X_r/Y}\to 0$ and $I/I^2\xrightarrow{\beta}
  \Gamma_{X/X_r/S}\to 0$, such that $\delta \circ \beta = \alpha$.  To
  conclude that $\Gamma_{X/Y/S}=0$ it suffices to see that $\alpha$ is
  injective, so $\delta$ is an isomorphism.  First, $X/Y$ is smooth at
  all associated points, implying $\Ker (\alpha)$ is $0$ at all such
  points. Second, since $X/Y$ is a l.c.i., so $I$ is locally generated
  by a regular sequence and $I/I^2$ is locally free, it follows that
  $\Ker (\alpha)=0$.

  (2): If $X/S$ is d.c.i. then since $i^*(\Omega_{X_r/S})$ is locally
  free, it follows that $\Gamma_{X/X_r/S}$ is locally free. Since
  $\pi$ is smooth at the associated points, it follows that
  $\Gamma_{X/Y/S}$ is $0$ at such points, and since $\Gamma_{X/Y/S}
  \subset \Gamma_{X/X_r/S} $ the assertion follows.
\end{proof}

We have the following composition and decomposition properties.
\begin{theorem}\label{comp-dci}
  Let $X\xrightarrow{f}Y \xrightarrow{g}S$ be a composition of
  dominant morphisms, locally of finite type.
  \begin{enumerate}[label=(\theenumi),ref=(\theenumi)]
  \item If $f$ is l.c.i., $g$ is d.c.i, and $f$ is smooth at all
    associated points of $X$, then $g\circ f$ is d.c.i.
\item Assume that $f$ is generically smooth and $g$ is smooth.  Then
    $g\circ f$ is d.c.i. if and only if $f$ is d.c.i.
\item Assume that $g$ and $g\circ f $ are l.c.i., $f$ is smooth at
  points that map to maximal points of $Y$, and $g$ is smooth at
  maximal points of $f(X)$. Then $\pdo \Omega_{X/Y,x}\leq 2$, $x\in
  X$.
\item Assume that $X$ and $Y$ are Cohen-Macaulay, $f$ is flat and
  generically smooth along each fibre.  Consider the conditions:
  \begin{enumerate}
  \item $g\circ f$ is d.c.i.
  \item $f $ is d.c.i., and $g$ is d.c.i. at all points in $f(X)$.
 \end{enumerate}
  Then $(a)\Rightarrow (b)$, and if $f$ is moreover l.c.i., then
  $(b)\Rightarrow (a)$.
\end{enumerate} 
\end{theorem}

\begin{remark}
  In \cite[(5.6), (5.7)]{avramov:lci} there are results that can be
  compared to \Theorem{\ref{comp-dci}}, using
  \Theorem{\ref{ferrand-vasconcelos}}. Making the added assumption
  that $\Oc_{X,x}$ be of finite flat dimension over $\Oc_{Y,\pi(x)}$
  for each $x\in X$, Avramov gets a stronger decomposition property
  for l.c.i. morphisms, namely that $f$ and $g$ are l.c.i. if and only
  if $g\circ f$ is l.c.i. This result depends on two other results
  that are fundamental albeit hard to get: (i) The vanishing of all
  higher Andr\'e-Quillen homology groups characterises
  l.c.i. morphisms; (ii) a certain connecting morphism of the
  Zariski-Jacobi long exact sequence is trivial
  \cite[(4.7)]{avramov:lci}.  Note however that (3) does not rely on
  flatness.
\end{remark}

\begin{proof} (1): By \Proposition{\ref{gamma0}}, $\Gamma_{X/Y/S}=0$.
  Since $X/Y$ is smooth at an associated point $\xi$, hence is
  flat at $\xi$, it follows that $Tor^1_{\Oc_{Y,\pi(\xi)}}(\Oc_{X,
    \xi},\Omega_{Y/S,\pi(\xi)})=0$. This implies that $\pdo
  \pi^*(\Omega_{Y/S})_x\leq 1$ for each point $x\in X$, since $g$ is
  d.c.i.; see also the proof of \Lemma{\ref{diffcomplete-lem}}. Since
  $\pdo \Omega_{X/Y,x}\leq 1$, it follows from the exact sequence
  (\ref{eq:can-ex}) that $\pdo \Omega_{X/S,x}\leq 1$.

  (2): Since $X/Y$ is smooth at all associated points in $X$ and since
  $f^*(\Omega_{Y/S})$ is locally free, it follows that
  $\Gamma_{X/Y/S}=0$.  Again we get the exact sequence
  \begin{equation}\label{exact-sequence}
    0 \to f^*(\Omega_{Y/S}) \to \Omega_{X/S} \to \Omega_{X/Y}\to 0,
  \end{equation}
implying that $ X/S$ is d.c.i. if and only if $X/Y$ is d.c.i..

  (3): By \Proposition{\ref{gamma0}}, $\Gamma_{X/Y/S}=0$, so we again
  have the  exact sequence (\ref{exact-sequence}). The assertion then follows from
  \Lemma{\ref{diffcomplete-lem}}.

  (4): First note that $X$ being Cohen-Macaulay, all its associated
  points are maximal, and that both (a) and (b) implies
  $\Gamma_{X/Y/S}=0$, by \Proposition{\ref{gamma0}}.  Let $y$ be a
  point in $Y$ and $j: X_y \to X$ and $p: X_y \to \Spec k_{Y,y}$ be
  the canonical morphisms. Then $j^*(\pi^*(\Omega_{Y/S})) = p^*
  (k_{Y,y}\otimes_{\Oc_{Y,y}}\Omega_{Y/S})$ is free, and since
  $j^*(\Omega_{X/Y})= \Omega_{X_y/k_{Y,y}}$ is generically free, while
  $X_y$ contains no embedded points since $X_y$ is Cohen-Macaulay, for
  $X$ and $Y$ are Cohen-Macaulay and $X/Y$ is flat. Therefore we get
  the exact sequence
  \begin{displaymath}
    0 \to j^*(\pi^*(\Omega_{Y/S})) \to j^*(\Omega_{X/S})\to
    \Omega_{X_y/k_{Y,y}}\to 0,
  \end{displaymath}
  hence $\pdo \Omega_{X_y/k_{Y,y},x} = \pdo j^*(\Omega_{X/S})_{x}$ when
  $x\in X_y$.  Now by \Lemma{\ref{diffcomplete-lem}}, $ \pdo
  j^*(\Omega_{X/S})_{x}$ $\leq 1$, hence by \Proposition{\ref{drci-flat}}
  $\pdo\Omega_{X/Y,x}\leq 1$ when $x\in X$.  The exact sequence
  (\ref{exact-sequence}) then implies that $g$ is d.c.i.  This proves
  (a)$\Rightarrow$ (b).  (b)$\Rightarrow$
 (a)  follows from (1).
\end{proof}

\section{Purity for general $d_{X/Y}$}\label{sec-pur-general}
For a coherent $\Oc_X$-module $M$ with local presentation $\Oc^m_X
\xrightarrow{\phi} \Oc_X^n \to M \to 0$ we let $I_i(\phi)$ be the
$i$th determinant ideal of $\phi$, assuming $i\leq \min\{m,n\}$, and
the ith Fitting ideal $F_i(M)= I_{n-i}(\phi)$. It is a fundamental
problem to compute the height of $F_i(M)$ for a given $M$, where an
important step is to find an upper bound in the form of refined
versions of Krull's principal ideal theorem. This problem has been
addressed in the literature, where bounds are determined in terms of
the integers $n,m$, the rank of $\phi$ (more precisely the analytic
spread of $M$ plays a role together with some measure of the
singularity defect of $X$), but one should remember that general
bounds of the height in terms of $m$ and $n$ are can be crude for a
given $M$, and will also depend on the presentation of $M$.  The
bounds that we will use arise from minimal presentations and from
exact sequences, based on either the classical Eagon-Northcott formula
involving the first two Betti numbers, but not the Euler number
(generic rank), or a refinement due to Eisenbud, Ulrich, and Huneke,
where the Euler number and the first Betti number appear, but $X$ is
assumed to be regular.


Assume for simplicity that $X$ is an integral scheme. Let
$\beta_{i}(M)= \sup \{\beta_{i}(M_x) \vert \ x\in X\}$ where
$\beta_i(x)= \beta_i(M_x) = \dim_{k_x} Tor^i_{\Oc_{X,x}}(k_x, M_x)$ is
the ith Betti number of the $\Oc_{X,x}$-module $M_x$. When $\pdo M_x <
\infty$ we can define the Euler number by $ \chi(x) =\chi(M_x)= \sum
(-1)^i\beta_i(x)$. Let $\cdots \to G^1_x\to G^0_x \to M_x\to 0$ be a
minimal free resolution, so $\beta_i(x)= \rank G^i$, and the partial
Euler numbers are defined by $\chi_i(x)=\chi_i(M_x)= \sum_{j\geq i}
(-1)^j\beta_j(x)$.  Note that in general $\chi_i(x)\neq \chi_i(\xi)$
when $x$ is a specialisation of the point $\xi$ and $i\geq 1$, while
$\chi(x)= \chi_0(x)= \chi(\xi)$ is the generic rank of $M_x$ and
$(-1)^i\chi_i(x)$ is the generic rank of the $i$th syzygy in a minimal
free resolution of $M_x$, so they are positive integers (see \cite[Th
19.7]{matsumura}).

Since $\beta_0(x)-\beta_1(x)= \chi(x)-\chi_2(x)$ and
$\chi_2(x)= (-1)^2 \chi_2(x)\geq 0$ it follows that 
\begin{displaymath}
  \gamma(x):=  \beta_0(x)-\beta_1(x)\leq \chi(x),
\end{displaymath}
with equality if $\pdo M_x \leq 1$, since then $\chi_2(x)=0$.

\begin{theorem}\label{fittingtheorem}
  \begin{enumerate}\item\label{eagon-northcott}
    If $i < \beta_0(x)- \min\{\beta_0(x), \beta_1(x)\}$, then
    $F_i(M_x)=0$, and if $F_i(M_x)\neq 0$ then
    \begin{displaymath}
      \hto (F_i(M_x)) \leq  (i+1)(i+1- \gamma(x)).
    \end{displaymath}
  \item\label{regular-bound} Assume $X$ is regular. Then
    \begin{displaymath}
      \hto(F_i(M_x))\leq
      \begin{cases}
        (i+1) (i+1-\chi(x))+ \beta_0(x)-i-1, \quad &\gamma(x)\geq 0,\\
        ( i+1 - \gamma(x)) (i+1-\chi(x))+ \beta_0(x)-i-1, \quad &\gamma(x) \leq 0.
      \end{cases}
    \end{displaymath}
  \item Let $0 \to M^1 \to M^2 \to M \to 0$ be a short exact sequence of
    coherent $\Oc_X$-modules.
    \begin{enumerate}\item\label{short-exact-height}
      \begin{displaymath}
        \hto (F_i(M_x))   \leq (i+1)(i+1+\beta_0(M^1_x) +\beta_1(M^2_x)- \beta_0(M^2_x))
      \end{displaymath}
    \item If $X$ is regular, one can replace $\chi (x)$ by $\chi(M_x^2)-\chi(M_x^1)$ in (2).
    \end{enumerate}
\end{enumerate}
\end{theorem}
\begin{remark}
  \begin{enumerate} 
  \item Assume that $M_x$ has generic rank $r= \beta_0(\xi)= \chi(x)$ and
    consider the Fitting ideal $F_r(M_x)$. If $ \gamma (x)\geq 0$, (2)
    strengthens (1) when
    \begin{displaymath}
      \beta_0(x) \leq       (r+1) ( \chi_2(x)+1)  \Leftrightarrow         r \geq -1 + \frac{\gamma(x)}2
      + \sqrt{\beta_0(x) + \frac{\gamma(x)^2}4}.
    \end{displaymath}
    If $\pdo M_x\leq 1$ the condition is $\beta_0(x)\leq r+1$, which is
    uninteresting since $F_r(M_x)= \Oc_{X,x}$ when
    $r \geq \beta_0(x)$. If the rank of the second syzygy
    $\chi_2(x)\neq 0$, then (2) is stronger (1) for not too high
    $\beta_0(x)$ or not too low $r$. If $\gamma (x)>0$. Then (2) strenghtens
    (1) if $\beta_0(x) < \frac {\gamma(x)^2}4$ för all $r$. If
    $\beta_0(x) \geq \frac {\gamma(x)^2}4$, then (2) improves (1) when
    \begin{displaymath}
      r \geq \frac{\gamma(x)}2 + \sqrt{\beta_0(x)- \frac{\gamma(x)^2}4}-1
    \end{displaymath}

  \item It is tempting, in the light of (1), to ask under what
    conditions $\hto(F_i(M_x))\leq (i+1)(i+1-\chi (M_x))$ holds when
    $\pdo M _x < \infty $.
\item \ref{short-exact-height} is stronger than \ref{eagon-northcott} when
  \begin{displaymath}
    \chi_2(M_x) \geq \chi_2(M^2_x)-\chi_1(M^1_x).
  \end{displaymath}
 \end{enumerate}
\end{remark}
\begin{proof}
  (1): Let $I_i(x)$ be the $i$th determinant ideal of the homomorphism
  $\phi_x : G^1_x \to G^0_x$, so $F_i(M_x)= I_{\beta_0(x)-i}$. That
  $F_i(M_x)=0$ when $i < \beta_0(x)- \min\{\beta_0(x), \beta_1(x)\}$
  follows since $\rank G^1 = \beta_1(x)$ and $\rank G^0 =
  \beta_0(x)$. The Eagon-Northcott bound \cite{eagon-northcott:height}
  (see also \cite[Th. 3.5]{buchsbaum-rim:II} and \cite[Th.
  13.10]{matsumura})
 \begin{displaymath}
  \hto (I_i)\leq (\min \{m,n\} -i +1)(\max \{m,n\} -i +1),
\end{displaymath}
now implies 
  \begin{displaymath}
    \hto (I_i) \leq (\beta_1(x) -i+1) (\beta_0(x)-i+1).
  \end{displaymath}

  (2): If $X$ is regular, the Eisenbud-Ulrich-Huneke bound \cite[Th.
  A]{eisenbud-huneke-ulrich:heights} implies
  $\hto(I_i)\leq ((-1)^1\chi_1(x)-i+1)(\max\{\beta_0(x), \beta_1(x)\}-i+1) +i-1 $.
  We have therefore two cases. If $\gamma(x) \geq 0$ then
  \begin{align*}
    \hto   F_i(M_x)&\leq (i+1-\chi(x)) (i+1)+ \beta_0(x)-i-1
  \end{align*}
  and if $\gamma(x) \leq 0$, then
  \begin{align*}
    \hto   F_i(M_x)& \leq (i+1-\chi(x)) ( i+1- \gamma(x))+ \beta_0(x)-i.
  \end{align*}
  
  (3):\thetag{a} There is an exact sequence
  $\Oc_{X,x}^{\beta_1(M^2_x)}\to \Oc_{X,x}^{\beta_0(M^2_x)}\to M^2_x\to 0$ and a
  surjection $\Oc_{X,x}^{\beta_0(M^1_x)}\to M^1_x\to 0$. Therefore we can
  assume $m\geq \beta_0(M^1_x)+\beta_1(M^2_x)$ and
  $n\leq \beta_0(M^2_x)$ in the Eagon-Northcott bound, hence
  $\hto (F_i(M_x))= \hto (I_{n-i})\leq (\min \{m,n\} -(n-i) +1)(\max
  \{m,n\} -(n-i) +1)= (i+1)(m-n+i+1) \leq (i+1)(i+1+\beta_0(M^1_x)
  +\beta_1(M^2_x)- \beta_0(M^2_x)) $. \thetag{b} is evident.
\end{proof}

The following interpretation of \Theorem{\ref{fittingtheorem}} is useful:

\begin{corollary}\label{locfreefitting} Let $X$ be an integral N\oe therian
  scheme and $M$ be a coherent $\Oc_X$-module. Let $j: U
  \hookrightarrow X $ be an open subset and put $V= X\setminus
  U$. Assume either of the conditions:
  \begin{enumerate}[label=(\theenumi),ref=(\theenumi)]
  \item $\codim_X V > (\sup_{x\in V}(\beta_1(M_x)- \beta_0(M_x))+
    \rank M+1)(\rank M +1)$.
  \item The projective dimension $\pdo M_x \leq 1$ at each point $x\in
    V$ and $\codim_X V > \rank M +1$.
  \end{enumerate}
  If $j^*(M)$ is locally free, it follows that $M$ is locally free. If
  $X$ moreover satisfies $(S_2)$, then $M= j_*j^*(M)$.
\end{corollary}
\begin{remark} In \cite[\S21.13]{EGA4:4} a couple $(X,V)$ is
  parafactorial if for any open set $\Omega$ the restriction functor
  $\Lc_\Omega \to \Lc\vert_{\Omega\cap U}$ is an equivalence of
  categories of invertible sheaves. In particular, this holds when
  $X=\Spec A$, $A$ is factorial of dimension $\geq 2$ and $ V=
  \{\mf_A\}$ \cite[21.6.13]{EGA4:4}.  By Grothendieck's finiteness
  theorem, if $U$ is Cohen-Macaulay and $j^*(M)$ is locally free, then
  $j_*j^*(M)$ is coherent when $\codim_X V \geq 2$. Assuming $X$ is
  Cohen-Macaulay, $\dim X\geq 2$, and $\codim_X V \geq 2$, we see that
  $j^*(M)$ is the restriction of a locally free sheaf if the maximal
  extension $j_*(j^*(M))$ satisfies either of the conditions (1) or
  (2) in \Corollary{\ref{locfreefitting}}.
\end{remark}
\begin{proof} The last assertion is evident so we have to prove that
  $M$ is locally free when $j^*(M)$ is locally free.  Let $
  \Oc_{X,x}^{m} \xrightarrow{\phi_x} \Oc_{X,x}^{n} \to M_x \to 0$ be a
  presentation at a point $x\in V$, where $m=\beta_1(M_x)$ and
  $n=\beta_0(M_x)$, and let $F_r(M_x)\subset \Oc_{X,x}$ be the $r$th
  Fitting ideal of $M_x$, where $r= \rank M$.  According to the
  Eagon-Northcott bound one gets as in the proof of
  \Theorem{\ref{fittingtheorem}} that $ \hto (F_r(M)) = \hto
  (I_{n-r})\leq (m-n+r+1)(r+1)$.  If $\codim_X V > (m-n+r+1)(r+1) =
  (\sup_{x\in V}(\beta_1(M_x) -\beta_0(M_x)) +\rank M +1)(\rank M +1)$
  it follows that $F_{r}(M)= \Oc_X$. This proves (1).  Assuming $\pdo
  M_x \leq 1$ it follows that the map $\phi_x$ is injective, hence
  $\rank M= \beta_0(M_x)-\beta_1(M_x)$ for each point $x$, implying
  (2).
\end{proof}

The relative embedding dimension of $X/S$ is

\begin{align*}
  \ed_{X/S} &=\beta_0(\Omega_{X/S})= \sup \{\dim_{k_x} k_x\otimes \Omega_{X/S} \ \vert \  x\in X\}\\
  &= \sup \{\embdim_{\Spec k_s} X_s \ \vert \ s\in S \},
\end{align*}
and the ``smoothness defect'' $\delta_{X/S} = \ed_{X/S} - d_{X/S}$ (this
is the regularity defect when $S$ is the spectrum of a perfect field
$k$ and $X= \Spec A$ for a local $k$-algebra $A$).  We also have the
dual notion of defect of tangent space dimension
\begin{displaymath}
  \eta_{Y/S} = \beta_0(T_{Y/S})=\sup \{ \dim_{k_y} k_y\otimes_{\Oc_{Y,y}} T_{Y/S} \ \vert
  \ y \in Y\} - d_{Y/S}.
\end{displaymath}
Clearly, $\eta_{Y/S}\geq 0$, and $\eta_{Y/S}=0$ if and only if
$T_{Y/S}$ is locally free and $Y/S$ is generically smooth.

The image of the tangent morphism is denoted
\begin{equation}\label{im-tan}
  \overline T_{X/S}=\Imo (d\pi: T_{X/S} \to T_{X/S \to
    Y/S} ).
\end{equation}

\begin{theorem}\label{purity of critical} Let $\pi: X/S \to Y/S$
  be a generically smooth morphism of integral N\oe therian
  $S$-schemes.
  \begin{enumerate}[label=(\theenumi),ref=(\theenumi)]
  \item\label{no1} $C_\pi \subset B_\pi$, and if $\Omega_{X/S}$ is
    locally projective, then
    \begin{displaymath}
      \Cc_{X/Y} =   Ext^1_{\Oc_X}(\Omega_{X/Y}, \Oc_X).
    \end{displaymath}
    (see also \Remark{\ref{remark:transpose}})
  \item\label{no2}
    \begin{enumerate}\item 
      If $X/S$ and $Y/S$ are locally of finite type, and $\pi$ is
      locally of finite type, then:
      \begin{displaymath}
        \codim^+_{X} B^{(i)}_\pi \leq      (d_{X/Y}+i+1)(i+1+\chi_2(\Omega_{X/Y}))
      \end{displaymath}
      and
      \begin{gather} \notag \codim^+_{X} B^{(i)}_\pi  \leq
        (d_{X/Y}+i+1)( d_{X/Y}-\rank \Omega_{X/S} +i+1+
        \beta_0(\Vc_{X/Y/S})+ \beta_1(\Omega_{X/S}) ) \\ \notag
        \leq(d_{X/Y}+i+1)(i+1+ \beta_0(\Vc_{X/Y/S})+
        \beta_1(\Omega_{X/S}) -d_{Y/S} ) \\ \tag{A} \leq
        (d_{X/Y}+i+1)( \delta_{Y/S}- \beta_0(\Gamma_{X/Y/S}) +
        \beta_1(\Vc_{X/Y/S})+\beta_1(\Omega_{X/S})+i+1).
      \end{gather}
      The second inequality is an equality when $X/S$ is generically
      smooth.
\item 
    \begin{displaymath}
\codim^+_{X} C^{(i)}_\pi \leq      (d_{X/Y}+i+1)(i+1+\chi_2(\Cc_{X/Y}))
    \end{displaymath}
and
   \begin{gather}\notag
      \codim^+_{X} C^{(i)}_\pi \leq (i+1) (\beta_0({\overline
        T}_{X/S})
      + \beta_1(T_{X/S \to Y/S}) - \rank (T_{X/S \to Y/S})
      \notag+i+1)\\ \notag \leq (i+1)( \beta_0({\overline T}_{X/S}) +
      \beta_1(T_{X/S \to Y/S}) - d_{Y/S}+i+1)\\ \tag{B} \leq (i+1)(
      \eta_{X/S} - \eta_{Y/S} + \beta_1({\overline T}_{X/S}) +
      \beta_1(T_{X/S \to Y/S}) +i+1).
    \end{gather}
    The second inequality is an equality when $Y/S$ is generically
    smooth.

    \end{enumerate}
 \item \label{no3} Assume that $\Omega_{X/S}$ and $\Cc_{X/Y}$ are
    coherent.  If $\overline T_{X/S}$ satisfies $(S_2)$, then
    $\codim^+_X C_\pi \leq 1$.  Assume moreover that $X$ is regular in
    codimension $\leq 1$ and $\Omega_{X/S,x}$ is free when
    $\hto(x)\leq 1$. Then the maximal points of height $1$ in $C_\pi$
    and $B_\pi$ coincide.
  \item\label{no4} Assume that $X/S$ and $Y/S$ are smooth. If
    $\overline {T}_{X/S}$ satisfies $(S_2)$, then
    \begin{displaymath}
\codim^+_X B_\pi    \leq 1.
  \end{displaymath}

  \item \label{no5} Assume $\pi$ is d.c.i., locally of
    finite type, and generically smooth. Then
    \begin{displaymath}
      \codim^+_X B^{(i)}_\pi \leq (d_{X/Y}+i+1)(i+1).
    \end{displaymath}
%
  \end{enumerate}
\end{theorem}

\begin{remark}
  \begin{enumerate}
  \item \ref{no5} is due to Dolgachev \cite{dolgachev:nonsmooth} (when
    $i=0$).  If $\pi$ is generically separably algebraic, then
    $\ref{no5} \Rightarrow \ref{no4}$, by \Lemma{\ref{drsci-smooth}}.
  \item By (3), $\codim^+_X C_\pi \leq 1$ essentially follows when
    Serre's $(S_2)$-property holds for the $\Oc_X$-modules $\overline
    T_{X/S}$ and $\Oc_X$.  If $X/Y$ is generically separably
    algebraic, so $\overline T_{X/S} = T_{X/S}$, then if $\Oc_X$
    satisfies ($S_2$) the module $\overline{ T}_{X/S}$ also satisfies $(S_2)$. We get
    $\codim^+_X B_\pi\leq 1$ in this case only when $X/S$ and $Y/S$
    are smooth, while \ref{no5} gives this for any d.c.i.
  \item If $X/S$ is smooth, then $\beta_1(\Omega_{X/S})=0$. If $Y/S$
    is smooth, then $\beta_1(T_{X/S \to Y/S})=0$. It follows from the
    proof that the last inequalities \thetag{A} and \thetag{B} can be
    improved by using higher Betti numbers, but the formulas are
    perhaps sufficiently long already.
  \end{enumerate}
\end{remark}

  \begin{pfof}{\Theorem{\ref{purity of critical}}} \ref{no1}: 
    Combining the dual of the exact sequences (\ref{eq:can-ex1},
    \ref{eq:can-ex2}), noting that $\Gamma_{X/Y/S}^*=0$ since $\pi$ is
    generically smooth, one gets the exact sequence
    \begin{displaymath}
      0 \to \Cc_{X/Y}\to    Ext^1_{\Oc_X}(\Omega_{X/Y},\Oc_X)\to Ext^1_{\Oc_X}(\Omega_{X/S},\Oc_X),
    \end{displaymath}
    implying $C_\pi \subset B_\pi$; it also implies the other
    assertion when $\Omega_{X/S}$ is locally projective (by
    quasi-coherence).

    \ref{no2}: \thetag{a}: Apply \Theorem{\ref{fittingtheorem}} (1) first
    to the module $M=\Omega_{X/Y}$, considering the Fitting ideal
    $F_{i+d_{X/Y}}(\Omega_{X/Y})$, to get the first inequality, then apply
    it to the exact sequence \eqref{eq:can-ex1} to get the first
    inequality in \thetag{A}. The second inequality follows since
    $-\rank \Omega_{X/S} + d_{X/Y}\leq -d_{X/S} + d_{X/Y}= -d_{Y/S}$, and this
    is an equality when $X/S$ is generically smooth since then
    $\rank \Omega_{X/Y}=d_{X/Y}$. The last inequality follows after
    considering the Betti numbers of the members in the exact sequence
    \eqref{eq:can-ex1}, giving
    $\beta_0(\Vc_{X/Y/S}) \leq \beta_0(\Omega_{Y/S}) - \beta_0(\Gamma_{X/Y/S}) +
    \beta_1(\Vc_{X/Y/S}) = \embdim_S Y - \beta_0(\Gamma_{X/Y/S}) + \beta_1(\Vc_{X/Y/S})
    = \delta_{Y/S}+ d_{Y/S} - \beta_0(\Gamma_{X/Y/S}) + \beta_1(\Vc_{X/Y/S})$.
    \thetag{b}: Again apply \Theorem{\ref{fittingtheorem}} but this
    time to the short exact sequence one gets from \eqref{fund-exact}
    (using $\overline T_{X/S} $). This immediately gives the first
    inequality. The first inequality in \thetag{B} follows since
    $\rank T_{X/S\to Y/S} = s_{Y/S} \geq d_{Y/S}$, and this is an equality
    when $Y/S$ is generically smooth. To get the last inequality
    consider the Betti numbers of the members in the exact sequence,
    giving
    $\beta_0(\overline T_{X/S}) \leq \beta_0(T_{X/S}) - \beta_0(T_{X/Y}) +
    \beta_1(\overline T_{X/S}) = \eta_{X/S} + d_{X/S} - \eta_{X/Y}- d_{X/Y} +
    \beta_1(\overline T_{X/S}) = \eta_{X/S} - \eta_{X/Y} +d_{Y/S} +
    \beta_1(\overline T_{X/S})$.

    \ref{no3}: Clearly $\codim^+_X C_{\pi}\leq 1 $ follows if
    $\Cc_{X/Y}$ has no associated points of height $\geq 2$.  Suppose
    the contrary, that there exists an associated point $x$ of height
    $\geq 2$, so $k_x \subset \Cc_{X/Y,x}$.  Letting $T^a_{X/S\to
      Y/S,x}$ be the pre-image of $k_x$ in $ T_{X/S\to Y/S,x}$ the
    short exact sequence
    \begin{equation}\label{test-seq} 0\to \overline T_{X/S,x}\to
      T_{X/S\to Y/S,x} \to \Cc_{X/Y,x} \to 0
    \end{equation} pulls back to 
    \begin{displaymath} 0\to \overline T_{X/S,x}\to T^a_{X/S\to Y/S,x}
      \to k_x \to 0.
    \end{displaymath}  Since $\overline T_{X/S,x}$ satisfies $(S_2)$,  so $Ext^1_{\Oc_{X,x}}(k_x, \overline T_{X/S,x})=0$,  the above
    sequence is split exact; hence $T^a_{X/S \to Y/S,x}\subset T_{X/S\to Y/S,x}$ has non-zero torsion. But
    $\Oc_{X,x}$ is integral, hence $T_{X/S\to Y/S,x} $ $ =
    Hom_{\Oc_{X,x}}$ $(\pi^*(\Omega_{Y/S})_x, \Oc_{X,x})$ is torsion free, which gives a contradiction.

    To see that points of height $1$ in $C_\pi$ and $B_\pi$ are equal,
    by \ref{no1} it suffices to see that if $x$ is a maximal point of
    $B_\pi$ of height $1$, then it belongs to $ C_\pi$.  Since $X/S$
    is smooth in codimension $\leq 1$ the dual of the exact sequence
    (\ref{eq:can-ex2}) (or apply \ref{no1} again) implies the first
    equality in
    \begin{displaymath} \Cc^\pi_x = Ext^1_{\Oc_X}(\Omega_{X/Y},
      \Oc_X)_x = Ext^1_{\Oc_{X,x}}(\Omega_{X/Y,x}, \Oc_{X,x}) \neq 0;
    \end{displaymath} 
    the second equality follows since $\Omega_{X/Y}$ is
    quasi-coherent. The inequality sign follows since $x\in B_\pi$
    implies that $\Omega_{X/Y,x}$ is not free over the  regular
    ring $\Oc_{X,x}$ of global homological dimension $1$.

    \ref{no4}: Since $\Omega_{X/S}$ and $\Omega_{Y/S}$ are locally
    free, the assertion follows from \Proposition{\ref{dual-crit}} and
    \ref{no3}.

    \ref{no5}: If $\pdo \Omega_{X/Y,x}\leq 1$ then $\chi_2(\Omega_{X/Y,x})=0$, so the assertion follows from \ref{no2}. 
  \end{pfof}
  \begin{remark}
    Assume in \ref{no5} that $X/Y$ is l.c.i.  instead of d.c.i.  (see
    \Theorem{\ref{ferrand-vasconcelos}}). Then $\codim^+_X B_\pi \leq
    d_{X/Y}+1$ follows from the Eagon-Northcott formula, applied to
    the exact sequence $ 0 \to I/I^2 \to \Omega_{X_r/Y}\to
    \Omega_{X/Y}\to 0$, assuming, as one may, that $X/Y$ is regularly
    immersed in a smooth scheme $X_r/Y$, with defining ideal $I$.
  \end{remark}

 \begin{theorem}\label{purity-br-loc} Let $\pi : X/S\to Y/S$ be a
    generically smooth morphism of N\oe therian integral $S$-schemes
    such that $\Omega_{X/S}$ and $\Omega_{Y/S}$ are coherent. 
    \begin{enumerate}
    \item 
  \begin{displaymath}  
    \codim_X^+B_\pi^{(i)}\leq (d_{X/Y}+i+1)(i+1+ \delta_{ Y/S}+ \chi_2(\Omega_{X/S})).
  \end{displaymath}
In particular, if  $X/S$ is d.c.i., then 
  \begin{displaymath}
    \codim_X^+B_\pi^{(i)}\leq (d_{X/Y}+i+1)(i+1+ \delta_{ Y/S}).
  \end{displaymath}
\item Assume that $X/S$ and $Y/S$ are d.c.i.. Then
      \begin{displaymath}
        \codim^+_X B^{(i)}_\pi \leq ( d_{X/Y}+i+1)(i+1 +\beta_1(\Omega_{X/Y})).
      \end{displaymath}
\item Assume that $X/S$ is smooth and each restriction to fibres $X_s
  \to Y_s$, $s\in S$, is generically smooth. Then
      \begin{displaymath}
        \codim_X^+ B_\pi \leq \delta_{X/Y} + d_{X/Y} -1.
\end{displaymath}
    \end{enumerate}
 \end{theorem}

 \begin{pf}
(1):   We have the exact sequence
\begin{displaymath}
  \pi^*(\Omega_{Y/S})\to \Omega_{X/S}\to \Omega_{X/Y}\to 0.
\end{displaymath}
Noting that $\beta_1(\Omega_{X/S,x})-\beta_0(\Omega_{X/S,x}) =
-\chi(\Omega_{X/S,x})+ \chi_2(\Omega_{X/S,x})= - d_{X/S}+
\chi_2(\Omega_{X/S,x})$ and $\beta_0(\pi^*(\Omega_{Y/S})) \leq
\delta_{ Y/S} + d_{Y/S}$, \Theorem{\ref{fittingtheorem}}, (3), implies
the first assertion.  If $X/S$ is d.c.i. (if we like, by
\Proposition{\ref{gamma0}} we can then insert $0\to $ to the left in
the exact sequence), then $\chi_2(\Omega_{X/S})=0$, implying the
second assertion.

(2): By \Theorem{\ref{comp-dci}} $\pdo \Omega_{X/Y}\leq 2$,
   hence $\chi_2(\Omega_{X/Y,x})= \beta_2(\Omega_{X/Y,x})=
   \chi(\Omega_{X/Y,x})-\beta_0(\Omega_{X/Y,x})+
   \beta_1(\Omega_{X/Y,x})= -\delta_{X/Y,x}+ \beta_1(\Omega_{X/Y,x})\leq \beta_1(\Omega_{X/Y})
$,
   so the assertion follows from \Theorem{\ref{purity of critical}},
   (2).

   (3): Since $X/S$ is smooth, each fibre $X_s$ is regular, and since
   $X_s\to Y_s$ is generically smooth, so $\Omega_{X_s/Y_s}$ has generic rank
   $d_{X_s/Y_s}$. By \Theorem{\ref{fittingtheorem}, (2), $i=0$}, with
   $M= \Omega_{X_s/Y_s}$, we get
   $\codim^+_{X_s} B_{X_s/Y_s}\leq 1-d_{X/Y} + \beta_0(\Omega_{X_s/Y_s}) -1 \leq
   \beta_0(\Omega_{X_s/Y_s}) - 1 =d_{X_s/Y_s} + \delta_{X_s/Y_s}-1$. Since
   $\codim_X^+B_\pi \leq \sup\{ \codim^+_{X_s} B_{X_s/Y_s}\ $
   $ \vert \ s\in S \}$,
   $d_{X/Y}= \sup \{d_{X_s/Y_s} \ \vert \ s\in S\}$ (see {\it generalities})
   and $\delta_{X/Y} = \sup \{\delta_{X_s/Y_s} \ \vert \ s\in S \}$ the assertion
   follows.
%
  \end{pf}

  \section{Purity when $d_{X/Y}=0$} \label{reldim0}
  \Theorem{\ref{purity-br-loc}} and \ref{no4}, \ref{no5} in \Theorem{
    \ref{purity of critical}} contain sufficient conditions to imply
  $\codim^+_X B_\pi \leq 1 $ when the relative dimension
  $d_{X/Y}=0$. These results rely mainly on establishing that $\pi$ be
  d.c.i., while \Lemma{\ref{drsci-smooth}} and
  \Proposition{\ref{drci-flat}} give sufficient, but rather
  restrictive conditions ensuring this. We aim for more precise
  results when $d_{X/Y}=0$, which is possible since maximal associated
  points of $\Omega_{X/Y}$ of high height cannot occur if
  $\Omega_{X/S}$ is void of associated points of high height, and if
  moreover $\Vc_{X/Y/S}$ has depth $\geq 2$ at such points. On the one
  hand, it is quite difficult to find upper bounds on the height of
  the associated points of $\Omega_{X/S}$. For instance, the rather
  natural assumption that $X$ satisfies ($S_2$) and $(R_1)$ does not
  imply that $\Omega_{X/S}$ satisfies $(S_1)$, and hence that
  $\Omega_{X/S}$ is torsion free since $X$ is integral. On the other
  hand, by the Auslander-Buchsbaum formula, a point $x$ cannot be
  associated to $\Omega_{X/S}$ if
  $\pdo_{\Oc_{X,x}} \Omega_{X/S,x} < \depth \Oc_{X,x}$.

  Purity for {\it finite}\/ morphisms $\pi$ has been intensely
  studied. It started with Zariski \cite{zariski:purity} who proved
  that $\codim_X^+ B_\pi \leq 1$ when $X/k$ and $Y/k$ are of finite
  type over a perfect field $k$, $X$ is normal, and $Y$ regular.
  Nagata \cite{nagata:purity} proved $\codim_X^+ B_{\pi}\leq 1$ when
  $Y$ is regular and $X$ is normal.  Auslander \cite{auslander} gave a
  module theoretic proof of Nagata's result.  Grothendieck states
  purity in the following way \cite{SGA2}: A local N\oe therian ring
  $(A,\mf)$ is pure if the restriction map of \'etale covers $Et(\Spec
  A)\to Et(\Spec A \setminus\mf)$ is an equivalence of categories, and
  a scheme $Y$ is pure at a point $y$ if $( \Oc_{Y,y}, \mf_y)$ is
  pure. Grothendieck proves, using Lefshetz conditions on
  $\Oc_Y$-modules for $Y= \Spec A$, that if $A$ is a N\oe therian
  complete intersection of dimension $ \geq 3$, then $A$ is pure
  (regular N\oe therian rings of dimension $\geq 2$ are pure).  It
  would we interesting to find a proof of Grothendieck's theorem such
  that etale covers of $\Spec A \setminus \mf$ can be constructively
  extended to an etale cover of $\Spec A$, as one can do when $ A$ is
  regular of dimension $2$.  According to
  \Theorem{\ref{purity-br-loc}} this is equivalent to finding an
  extension $X\to \Spec A$ of an etale morphism $X_0 \to \Spec A
  \setminus \{\mf\}$ such that $X$ is d.c.i.

Grothendieck's theorem implies the following result.

\begin{proposition}\label{cutkosky}
  Let $\pi: X\to Y$ be a finite morphism of normal N\oe therian
  integral schemes, where $Y$ is a local complete intersection. Then
  \begin{displaymath}
    \codim^+_X B_\pi \leq 2.
  \end{displaymath}
\end{proposition}
A somewhat more involved proof of this assertion, assuming $Y$ is
excellent, is due to Cutkosky \cite[Th. 5]{cutkosky}.
\begin{proof}
  By normality, $\pi$ is formally \'etale if and only if it is
  formally unramified, i.e.  $\Omega_{X/Y}=0$.  If $x$ is a maximal
  associated point of $\Omega_{X/Y}$ we argue that $\hto (x)\leq 2$.
  Suppose $\hto (x)\geq 3$.  Points $x_1\in \Spec \Oc_{X,x} \setminus
  {\mf_x}$ specialise to the closed point $x$, and since $x$ is a
  maximal associated point it follows that $\Spec \Oc_{X,x} \setminus
  \{\mf_x\}$ contains no associated point for $ \Omega_{X/Y}$.
  Therefore $\Omega_{X/Y,x_1}= 0$ when $x_1 \in \Spec \Oc_{X,x}
  \setminus \{\mf_x\}$, i.e. the morphism $\Spec \Oc_{X,x}\setminus
  \{\mf_x\}\to \Spec \Oc_{Y,\pi(x)}\setminus \{\mf_{\pi(x)}\}$ is
  \'etale.  By Grothendieck purity described above it follows that
  $\Omega_{X/Y,x}=0$, contradicting the assumption that $x$ is an
  associated point.
\end{proof}

Let $D_\pi= \pi(B_\pi)$ be the discriminant set of a finite and
generically separable morphism $\pi: X\to Y$, so $\codim^+_Y D_\pi$ $
= \codim^+_X B_\pi$. Assume $Y$ is a closed subscheme of a regular
scheme $Y_r/k$ over a perfect field $k$.  Perhaps a starting point for
Faltings in \cite [Th.  2]{algebraisation:faltings} was that an
\'etale covering of $Y\setminus D_\pi$ extends to an \'etale covering
of a formal neighbourhood $\hat U$ of $Y\setminus D_\pi$ in $U=Y_r
\setminus D_\pi$ \cite[Exp. X, Prop. 1.1]{SGA2}. He shows that such an
\'etale covering, given by a coherent $\Oc_{\hat U}$-algebra, actually
comes by the completion of a coherent $\Oc_{Y_r}$-algebra $\Ac$
(possibly ramified) when $\codim_Y^-D_\pi$ is sufficiently high, and
notices that $\Ac$ is normal, so by Zariski-Nagata-Auslander purity it
is actually unramified.  The allowed size of $D_\pi$ is $\codim^-_Y
D_\pi \geq \delta_{Y/k} +2$ expressed in
\cite [Th.  2]{algebraisation:faltings}, where the regularity defect
is $\delta_{Y/k} =\sup \{\dim_{k_y} \mf_y/\mf_y^2 - \hto (y) \ \vert \
y\in D_\pi\}$.  This readily implies that if each maximal point $x$ in
$B_\pi$ satisfies $\hto (x) \geq \delta_{Y/k} +2$, then actually $B_\pi
= \emptyset$.  Therefore
\begin{displaymath}\tag{F-C}
  \codim_X^+ B_\pi \leq \delta_{Y/k} +1.
\end{displaymath}
This was observed by Cutkosky \cite[Th. 6]{cutkosky}.
\Theorem{\ref{purity-br-loc}, (1),} generalises this inequality to
positive relative dimensions $d_{X/Y}\geq 0$, making the added
assumption that $X/S$ be d.c.i., or (perhaps more generally)
$\chi_2(\Omega_{X/S})=0$. Compare to \Theorem{\ref{purity-br-loc},
  (3)},  which gives $ \codim_X^+ B_\pi \leq \delta_{X/Y}$ when $X/k$ is
smooth. We can make another comparison to
\Theorems{\ref{purity-br-loc}}{\ref{purity of critical}} , putting
$d_{X/Y}=0$, which improve \thetag{F-C} when either
\begin{align*}
  \delta_{Y/k} &> \delta_{Y/S} + \chi_2(\Omega_{X/S}), \quad   \text{ or } \\
\delta_{Y/k} & > \delta_{Y/S}- \beta_0(\Gamma_{X/Y/S}) +
  \beta_1(\Vc_{X/Y/S})+\beta_1(\Omega_{X/S}).
\end{align*}
The latter inequality holds for example when $Y/k$ is non-smooth
(i.e. $Y$ is non-regular) so $\delta _{Y/ k} >0$, while assuming $X/S$
is smooth, $\Gamma_{X/Y/S}= 0$, and $\delta_{Y/S} +
\beta_1(\pi^*(\Omega_{Y/S})) < \delta_{Y/ k} $ (``$Y/S$ is more smooth
than $Y/k$''). If $X/S$ and $Y/S$ are smooth, and $X/Y$ is generically
smooth, one gets $0$ in the right-hand side of the second inequality,
but a bound in this situation, however, also follows from \ref{no5} in
\Theorem{\ref{purity of critical}} and \Lemma{\ref{drsci-smooth}}.

The purity results in \cite{zariski:purity,nagata:purity,
  auslander,SGA2,cutkosky,algebraisation:faltings,kantorovitz} apply
to finite morphisms, while \Theorems{\ref{purity of
    critical}}{\ref{purity-br-loc}} apply  to a
much larger class of morphisms, although the conclusion is weaker than
in the purity results of Grothendieck-Cutkosky (and
Zariski-Nagata-Auslander-Faltings) in the finite case when the
morphism $X\to Y$ cannot be fibred over $S$ as stated ($X/S$ is
required to be d.c.i. to rule out the existence of associated points
of high height for $\Omega_{X/S}$). On the other hand, van der
Waerden's theorem \cite{EGA4:4}*{Th.  21.12.12} states that
$\codim^+_X B_\pi \leq 1$ when $\pi: X\to Y$ is locally of finite type
and {\it birational}, and $Y$ moreover satisfies the following
condition (see also \cite{brenner:affine}):
\begin{description}
\item [(W)] If $T $ is an irreducible closed subset with $\codim^+_Y T
  \leq 1$, then the inclusion morphism $Y\setminus T \to Y$ is affine.
\end{description} The condition ({\bf W}) is satisfied in particular
when $Y$ is normal and its local divisor class groups are torsion,
i.e.\ $Y$ is $\Qb$-factorial.

It is convenient to single out the following class $(\bf F)$ of
morphisms:

\begin{description}
\item[(\bf F)] $\pi:X/S\to Y/S$ is dominant, generically separably
  algebraic, and locally finite type,  and the schemes $X,Y$ and $S$ are integral and N\oe
  therian.
\end{description}

\begin{remark}\label{bir-fin} The remark \cite [Rem.
  21.12.14,(v)]{EGA4:4} contains a discussion of the possibility to combine
  van der Waerden purity with Zariski-Nagata-Auslander purity. This
  problem was resolved by Gabber when $Y$ is regular.
  \end{remark}
  \begin{theorem}[Gabber \cite{stacks-tag-0EA3}]\label{gabber}
    Let $f \colon X \to Y$ be a morphism essentially of finite type from
    a normal scheme $X$ to a locally noetherian regular scheme $Y$.
    Assume that $f$ is essentially étale at all points of $X$ of
    codimension $\leq 1$. Then $f$ is essentially étale.
  \end{theorem}
  We can try go further not assuming that $Y$ be regular, and will
  generalize in three theorems van der Waerden's theorem to certain
  morphisms in the class $(\bf F)$. \Theorem{\ref{example-th}} sheds
  more light on \Remark{\ref{bir-fin}} by showing what is needed to
  make the finite and birational purity theorems work together.
  Morally, one needs to know that $X$ is (close to) UFD or that it be
  d.c.i., but for the remaining results such conditions are replaced
  by depth conditions, and we make no reference to van der Waerden's
  theorem. \Theorem{\ref{zariski-rel}} can be thought of as a
  characteristic $0$ relative version of Zariski-Nagata-Auslander
  purity (smooth $Y/S$), generalised to morphisms $\pi$ that need not be
  finite, but induce algebraic extensions of function fields.
  \Theorem{\ref{simult}} applies to certain non-smooth bases $Y/S$
  also in positive characteristic, at the price of a higher depth
  assumption on the source $X$.

\begin{theorem}\label{example-th}
  Let $\pi: X/S \to Y/S$ be a morphism of the type $(\Fb)$.  Assume that $Y$ is pure at all points of height $\geq 2$, and
  that $X$ and $Y$ satisfies the condition $({\bf W})$. Then
  $\codim^+_X B_\pi \leq 1$.
\end{theorem}
\begin{proof}
  Suppose $x$ is a maximal point in $B_\pi$ of height $\geq 2$ in $X$,
  and put $y=\pi(x)$.  We then have $\hto(y) \geq 2$ by the dimension
  inequality.  Let $X(y)^0= \pi^{-1}(\Spec \Oc_y \setminus \{\mf_y\})$
  and $X(y)= \pi^{-1}(\Spec \Oc_y)$.  Since $Y$ is pure at $y$ the
  restriction of $\pi$ to the morphism $\pi(y): X(y)^0\to \Spec \Oc_y$
  can be extended to an etale morphism $E(y)\to \Spec \Oc_y$.  Then
  clearly $E(y)$ is birational to $X(y)$.  The graph of the birational
  correspondence gives birational maps $p_1: Z \to E(y)$ and $p_2: Z
  \to X(y)$ where $\codim_{E(y)} B_{p_1}\geq 2$ and $\codim_{X(y)}
  B_{p_2}\geq 2$.  Since $E(y)$ is etale over $\Spec \Oc_y$ and $Y$
  satisfies $({\bf W})$, it follows that $E(y)$ satisfies $({\bf W})$,
  hence by van der Waerden's theorem $p_1$ is an isomorphism.  By
  assumption $X$ satisfies ({\bf W}), hence again by van der Waerden's
  theorem $p_2$ is an isomorphism.  This implies that $\pi$ is smooth
  at $x$, in contradiction to the assumption. Therefore there exist no
  maximal points in $B_\pi$ of height $\geq 2$.
\end{proof} \begin{theorem}\label{zariski-rel} Let $\pi: X/S \to Y/S$
  be of the type $(\Fb)$, where moreover $Y/S$ is smooth, $X$
  satisfies $(S_2)$, and $S$ is defined over the rational
  numbers. Then $\codim^+_X B_\pi \leq 1$.
\end{theorem}
\begin{proof} If on the contrary there exists a maximal point of
  height $\geq 2$, after localisation one may assume that $\hto(x)\geq
  2$ for each maximal point $x$ in $B_\pi= \supp \Omega_{X/Y}$.  Since
  $\Omega_{Y/S}$ is locally free the exact sequence (\ref{eq:can-ex})
  can be complemented with $0 \to $ to the left; each maximal point of
  $B_\pi$ is of height $\geq 2$; $\pi^*(\Omega_{Y/S})$ is locally
  free, and $X$ satisfies $(S_2)$; hence (\ref{eq:can-ex}) is locally
  split exact; hence the quotient $\Omega_{X/S}/\Omega_{X/S}^t$ by the
  torsion sub-module is locally free.  It follows that $T_{X/S}$ is
  locally free (of finite rank) and that the canonical map
  $\Omega_{X/S}\to T_{X/S}^*$ is surjective (see e.g. \cite[Lem.
  2]{lipman:jacobian}). Therefore, by \cite{nagata:remark-zariski} (see\cite[\S3, p.
  880]{lipman:freeder}), for any point $x$ in $X$ there exist elements
  $\{f_1, f_2, \dots , f_r\}$ in $\Oc_{X,x}$ and derivations
  $\partial_1,
  \partial_2, \dots , \partial_r$ ($r= \rank
  \Omega_{X/S,x}/\Omega_{X/S,x}^t$) such that $\det \partial_i(f_j)$
  is invertible.  Imitating the proof of the Zariski-Lipman-Nagata
  regularity criterion; see \cite[]{lipman:freeder},\cite[Th 30.1 and
  its Corollary]{matsumura}, noting that the completion of the local
  ring $\Oc_{X,x}$ along the ideal $(f_1, f_2, \dots, f_r)$ is reduced
  since $X/S$ is locally of finite type and integral, it follows that
  $\Omega_{X/S,x}$ is free, so $\Omega_{X/S,x}^t=0$; hence
  $\Omega_{X/Y,x}=0$, contradicting the assumption that $x\in
  B_\pi$. Therefore $B_\pi $ contains no maximal points of height
  $\geq 2$.
\end{proof}

\begin{theorem}\label{simult} Let $\pi: X/S \to Y/S$ be a morphism of the type $(\Fb)$. Make also  the assumptions:
  \begin{enumerate}[label=(\theenumi), ref=(\theenumi)]
  \item $X$ satisfies $(S_3)$, $X/S$ is d.c.i., and $\codim^-_XB_
    {X/S}\geq 2$.
  \item $\Omega_{Y/S,y}$ is free for $y\in D_\pi$ such that $\hto
    (x)\leq 2$ when $y=\pi(x)$.
  \item\label{imperf-cond} $\pdo \Vc_{X/Y/S,x} \leq 1$ for each point
    $x$.
  \end{enumerate}
  Then $\codim^+_X B_\pi \leq 1$.
\end{theorem}

\begin{remark}
  Note that $\Gamma_{X/Y/S}= 0$ when $X/Y$ is locally of finite type and
  d.c.i. \Prop{\ref{gamma0}}, but this case is superseded by
  \Theorem{\ref{purity of critical}, \ref{no5}}. Hence
  \Theorem{\ref{simult}} is useful when one instead of knowing $X/Y$
  be d.c.i., which may be hard to get when $Y/S$ is not d.c.i. (see
  \Theorem{\ref{purity-br-loc}}). On the other hand, assume that
  $\pdo \pi^*(\Omega_{Y/S}) \leq 1$; this holds when $Y/S$ is d.c.i. and
  $F_1(\pi^*(\Omega_{Y/S})) = \Oc_XF_1(\Omega_{Y/S})$. Then \ref{imperf-cond}
  holds if $\Gamma_{X/Y/S}$ is locally free or if the canonical map
  $\delta: \Gamma_{X/X_r/S}\to \Gamma_{X/X_r/Y} $ is injective; the latter holds if
  $I/I^2 \to \Gamma_{X/X_r/Y} $ is injective; see the proof of
  \Proposition{\ref{gamma0}}.
\end{remark}
Note that since $\pi$ is generically smooth, $\Gamma_{X/Y/S}$ is
torsion, so $\Vc_{X/S}^* = T_{X/S\to Y/S}$.

\begin{pfof}{\Theorem{\ref{simult}}} Assume that $\codim^+_X B_\pi
  \geq 2 $, so after localisation we may assume that each maximal
  point $x$ if $B_\pi$ has height $\hto (x) \geq 2$.  Since $\pi$ is
  generically separable $Hom_{\Oc_X}(\Cc_{X/Y}, \Oc_X)$ $=0$, we have
  from (\ref{fund-exact}) the commutative diagram
  \begin{align*}\hskip -1.5cm \xymatrix{
      0\ar[r]&Hom_{\Oc_X}(T_{X/S \to Y/S}, \Oc_X) \ar[r]\ar[d]^\alpha & Hom_{\Oc_X}(T_{X/S}, \Oc_X)\ar[r]\ar[d]^\beta & \Lambda_{X/Y} \ar[r]\ar[d]^\gamma&0\\
      0\ar[r]& Hom_{\Oc_X}({\overline T}_{X/S}, \Oc_X)\ar[r]&
      Hom_{\Oc_X}(T_{X/S}, \Oc_X)\ar[r]& Hom_{\Oc_X}(T_{X/Y}, \Oc_X)}
  \end{align*}
  where ${\overline T}_{X/S} = \Imo (d\pi)$ and $\Lambda_{X/Y}$ is
  defined by the diagram.  Here $\beta$ is the identity and by the
  serpent lemma $ \Ker (\gamma) =\Coker (\alpha)=
  Ext^1_{\Oc_X}(\Cc_{X/Y}, \Oc_X)$. Hence we have the exact sequence
  \begin{displaymath}
    0 \to Ext^1_{\Oc_X}(\Cc_{X/Y}, \Oc_X) \to \Lambda_{X/Y} \to  Hom_{\Oc_X}(T_{X/S}, \Oc_X).
  \end{displaymath}

  As a maximal point $x$ of $C_\pi$ is maximal also in $B_\pi$
  \Th{\ref{purity of critical}}, we have $\hto (x)\geq 2$, hence
  $Ext^1_{\Oc_X}(\Cc_{X/Y}, \Oc_X)=0$, since $X$ satisfies $(S_2)$. It
  is now straightforward to see that we get the following commutative
  diagram (see \Remark{\ref{remark:transpose}}):
  \begin{align}\hskip -2cm \xymatrix{&
      Ext^2_{\Oc_X}(D(\Vc_{X/Y/S}),\Oc_X)&
      Ext^2_{\Oc_X}(D(\Omega_{X/S}), \Oc_X)&
      Ext^2_{\Oc_X}(D(\Omega_{X/Y}), \Oc_X)&
      \\
      0 \ar[r]& \ar[u]Hom_{\Oc_X}(T_{X/S \to Y/S}, \Oc_X)\ar[r]&\ar[u]
      Hom_{\Oc_X}(T_{X/S},\Oc_X)\ar[r]&
      Hom_{\Oc_X}(T_{X/Y},\Oc_X)\ar[u]&  \\
      0 \ar[r]& \Vc_{X/Y/S} \ar[r]\ar[u]& \Omega_{X/S} \ar[u]\ar[r]&
      \Omega_{X/Y}\ar[r]\ar[u]^h& 0 \\
      & Ext^1_{\Oc_X}(D(\Vc_{X/Y/S}), \Oc_X)\ar[u]&
      Ext^1_{\Oc_X}(D(\Omega_{X/S}), \Oc_X) \ar[u]&Ext^1_{\Oc_X}(
      D(\Omega_{X/Y}), \Oc_X).\ar[u]& }
  \end{align}
  Since $\pi$ is generically separably algebraic, so $h=0$, by the
  serpent lemma one has the exact sequence
  \begin{multline*} 0 \to Ext^1_{\Oc_X}(D(\Vc_{X/Y/S}), \Oc_X)
    \xrightarrow{a}
    Ext^1_{\Oc_X}(D(\Omega_{X/S}), \Oc_X)\to \Omega_{X/Y}\to \\
    \to Ext^2_{\Oc_X}(D(\Vc_{X/Y/S}), \Oc_X) \xrightarrow{b}
    Ext^2_{\Oc_X}(D(\Omega_{X/S}), \Oc_X),
  \end{multline*}
  whence $B_\pi = \supp \Coker a \ \cup\ \supp \Ker b$.  By
  \ref{imperf-cond}, $D(\Vc_{X/Y/S})$ is projectively equivalent to \\ $
  Ext^1_{\Oc_X}(\Vc_{X/Y/S}, \Oc_X)$, and since $\Omega_{Y/S,y}$ is
  free at points $y$ when $y=\pi(x)$ and $\hto(x)\leq 2$, and
  therefore $\Vc_{X/Y/S,x}$ is free when $\hto(x)\leq 2$. Therefore
  \begin{displaymath}
  \codim_X^- \supp D(\Vc_{X/Y/S}) \geq 3;
\end{displaymath}
hence
  \begin{displaymath}
    Ext^2_{\Oc_X}(D(\Vc_{X/Y/S}), \Oc_X) =0
\end{displaymath}
since $X$ satisfies
  $(S_3)$; hence $\Ker b = 0$.  Similarly, since $X/S$ is d.c.i. and
  $\Omega_{X/S,x}$ is free at points $x$ with $\hto(x)\leq 1$, we get
  \begin{displaymath}
    Ext_{\Oc_X}^1(D(\Omega_{X/S}), \Oc_X)=
    Ext_{\Oc_X}^1(Ext_{\Oc_X}^1(\Omega_{X/S}, \Oc_X), \Oc_X)= 0;
  \end{displaymath}
  hence $\Coker a = 0$. Therefore $\Omega_{X/Y}=0$.

  If $\Gamma_{X/Y/S}=0$, the last assertion in \ref{imperf-cond} in
  \Theorem{\ref{simult}} follows from \Lemma{\ref{diffcomplete-lem}}.
  (If $B_{Y/S}=\emptyset$, it is evident that $\Gamma_{X/Y/S}=0$.)
\end{pfof}
\begin{remark}
  If $X/Y$ is a d.c.i., $X$ satisfies $(S_3)$, and each maximal point
  of $B_\pi$ is of height $\geq 2$, then $\Omega_{X/Y}$ is reflexive
  (it is actually locally free \Th{\ref{purity of
      critical},\ref{no5}}).  This follows since $D(\Omega_{X/Y})$ is
  locally projectively equivalent to $ Ext^1_{\Oc_X}(\Omega_{X/Y},
  \Oc_X)$, implying $\codim^+_X D(\Omega_{X/Y}) \geq 3$ since
  $\Omega_{X/Y}$ is free at points of height $\leq 2$; hence
  $Ext^1_{\Oc_X}(D(\Omega_{X/Y}),\Oc_X)=Ext^2_{\Oc_X}(D(\Omega_{X/Y}),\Oc_X)=0$.
\end{remark}

We state an easily workable excerpt of the above results that may
prove useful to get that a morphism is \'etale.

\begin{corollary}\label{cor-pure} Let $\pi: X/S \to Y/S$ be a morphism of integral N\oe therian $S$-schemes
  which is locally of finite type generically separably algebraic and
  dominant. Assume moreover that $Y$ is normal, and that at points of
  height $\leq 1$, $X$ is regular and $X/S$ is smooth.  The following
  are equivalent:
  \begin{enumerate}[label=(\theenumi), ref=(\theenumi)]
  \item\label{o} $\pi$ is submersive.
  \item\label{oo} $\pi$ is submersive at points of height $1$.
  \item\label{ooo} there exists an exceptional set $E\subset X$ of
    codimension $\geq 2$ such that the restriction $X\setminus E \to
    Y\setminus \pi(E)$ is formally \'etale.
  \end{enumerate}
  Make one of the following additional assumptions:
  \begin{enumerate}[label=(\alph*)]
  \item $\pi$ is birational and locally of finite type, $X$ and $Y$
    are normal, and $Y$ satisfies the condition $(\Wb)$.
  \item $\pi$ is quasi-finite, $X$ is normal and $Y$ is regular.
  \item The assumptions in \Theorem{\ref{zariski-rel}} or
    \Theorem{\ref{simult}} hold.
  \item $\pi$ is a d.c.i. (see e.g.
    \Lemmas{\ref{drsci-smooth}}{\ref{drci-flat}}).
  \item $X/S$ and $Y/S$ are d.c.i and $\beta_1(\Omega_{X/Y})\leq 1$.
  \end{enumerate} Then it follows that $(1-3)$ are equivalent to:
  \begin{enumerate} [label=(4)]
  \item $\pi$ is \'etale.
  \end{enumerate}
\end{corollary}

\begin{remark}
  It is not sufficient to know $\codim^+_X B_\pi \leq 2$ to get (4)
  from \ref{o}, as seen from the following well-known example: If $\pi
  : X=\Spec k[s,t] \to Y=\Spec k[s^2, st, t^2] $, then $B_\pi =
  \{(s,t)\}$ and $C_\pi = \emptyset$. Here $Y/k$ is d.c.i.  and $X/k$
  is smooth, but $\pi$ is not d.c.i. Assume that all morphisms
  $X/S,Y/S$, and $X/Y$ are locally of finite type and assume that
  $\pi: X/S \to Y/S$ is generically smooth. Let $B_{X/S}$ be the
  branch locus of $X/S$, so $B_{X/S} \subset X^s$, where $X^s$ is the
  locus of points where $X/S$ fails to be smooth, and $X^s = B_{X/S} $
  when $X/S$ is flat \cite[\S17]{EGA4:4}.  It is easy to see that if
  $Y/S$ is smooth, then $B_{X/S}\subset B_\pi$; but $B_{X/S}$ need not
  be contained in $C_\pi$, by the above example.
\end{remark}

\begin{proof} $\ref{o}\Leftrightarrow \ref{oo}$: follows from (3) in
  \Theorem{\ref{purity of critical}}.  $\ref{ooo}\Rightarrow \ref{oo}$
  is clear.  $\ref{oo}\Rightarrow \ref{ooo}$: If $x\in X$ is a point
  of height $1$, then $\Omega_{X/Y,x}=0$ by (3) in
  \Theorem{\ref{purity of critical}}; since $Y$ is normal it follows
  that $\Oc_{Y,y} \to \Oc_{X,x}$ is \'etale when $\hto (x)\leq 1$
  \cite[Cor.  18.10.3]{EGA4:4}.  $(4)\Rightarrow \ref{ooo}$ is
  evident.  $\ref{o}\Rightarrow (4)$: Since $Y$ is normal, $\pi$ is
  \'etale if and only if $\Omega_{X/Y}=0$. By \ref{o} $C_\pi =
  \emptyset$, hence by \Theorem{\ref{purity of critical}, (3),}
  $B_\pi$ contains no points of height $1$. Therefore $\pi$ is \'etale
  if $\codim^+_X B_\pi \leq 1$.  In (a) this follows from van der
  Waerden's theorem, $(b)$ follows from the Zariski-Nagata-Auslander
  theorem, $(d)$ follows from (5) in \Theorem{\ref{purity of
      critical}}. $(e)$ follows from \Theorem{\ref{purity-br-loc}}.
\end{proof}

Stay with the case $d_{X/Y}=0$. It is natural to fix the general type
of morphism as in (\Fb) but one can make different assumptions on the
source $X$. Let $\Cc_S$ be a class of integral and N\oe therian
schemes over $S$.
\begin{definition}\label{pure-scheme} A N\oe therian scheme $Y/S$ is
  {\it strongly (weakly) $\Cc_S$-pure}\/ if for any morphism of the
  type $({\bf F})$ we have $\codim^+_X B_\pi\leq 1$ ($\codim^+_X C_\pi
  \leq 1$) when $X/S$ belongs to $\Cc_S$.
\end{definition}

Let $\Cc^{2}_S$ be the category of integral schemes $X/S$ locally of
finite type satisfying $(S_2)$, and $\Cc^s_{S}$ the category of smooth
schemes. For example, \Theorem{\ref{purity of critical}} implies that
all N\oe therian integral schemes $Y/S$ are weakly $\Cc^{2}_S$-pure,
and that smooth schemes are strongly $\Cc^s_{S}$-pure.

It should not be a big surprise that $\codim^+_X C_\pi \leq 1$ holds
under weaker conditions than those needed to get $\codim^+_X B_\pi
\leq 1$. For example, Griffith \cite{griffith:ram} has examples where
$B_\pi$ does not satisfy $\codim^+_X B_\pi \leq 1$ for finite
morphisms $\pi:X/S \to Y/S$ of the type $({\bf F})$, when $X$ is
normal and $Y$ Gorenstein, but we know that $Y/S$ is weakly
$\Cc^2_S$-pure, so $\codim^+_X C_\pi \leq 1$.

\bigskip \textsc{Department of Mathematics, University of G\"avle,
  S-801 76, Sweden}

{\it E-mail address}\/: \texttt{rkm@hig.se}.

\printbibliography
\end{document}